\documentclass[a4paper,oneside,final]{amsart}

\usepackage{geometry}

\usepackage[english]{babel}
\usepackage{graphicx}
\usepackage{wrapfig}
\usepackage{mathtools} 		
\usepackage{amsmath, amssymb, amsfonts, amsthm} 
\usepackage{hyperref}
\usepackage{url}            
\usepackage[dvipsnames]{xcolor}
\usepackage{color}         
\usepackage{changes}		
\usepackage{pdfpages}
\usepackage{enumerate}

\usepackage{cleveref}

\usepackage{ifdraft}

\usepackage{tikz}
\usetikzlibrary{calc,arrows.meta,bending}
\ifoptionfinal{}{%
	\usetikzlibrary{external}\tikzexternalize
}
\usepackage{pgfplots}
\pgfplotsset{compat=newest}

\def\centerarc[#1](#2)(#3:#4:#5)(#6);%
{
	\draw[#1]([shift=(#3:#5)]#2) arc (#3:#4:#5) node[below] { #6 };
}

\usepackage[subrefformat=parens]{subcaption}

\usepackage{booktabs}

\usepackage{rotating}

\usepackage{autonum}

\usepackage[nice]{nicefrac}
\hypersetup{colorlinks=true,urlcolor=blue,linkcolor=black}
\theoremstyle{plain}
\newtheorem{lem}{Lemma}[section]
\newtheorem{thm}[lem]{Theorem}
\newtheorem{cor}[lem]{Corollary}
\newtheorem{prop}[lem]{Proposition}

\theoremstyle{definition}
\newtheorem{defn}[lem]{Definition}
\newtheorem{rem}[lem]{Remark}

\newcommand{\NN}{\mathbb N}
\newcommand{\ZZ}{\mathbb Z}

\newcommand{\RR}{\mathbb R}

\newcommand{\sE}{\mathcal{E}}

\newcommand{\sA}{\mathcal{A}}
\newcommand{\sG}{\mathcal{G}}

\newcommand{\sL}{\mathcal{L}}

\newcommand{\sS}{{\mathbb{S}^1}}

\newcommand{\intd}[1]{\, \mathrm{d}#1}

\numberwithin{equation}{section}

\usetikzlibrary{angles,quotes}

\makeatletter
\@namedef{subjclassname@2020}{%
	\textup{2020} Mathematics Subject Classification}
\makeatother

\begin{document}
	
	\title[An area preserving elastic flow]{Dynamics of elastic wires:\\ Preserving Area Without Nonlocality}
	
	
	\author[L.~Langer]{Leonie Langer}
	\address[L.~Langer]{Institute of Applied Analysis, Ulm University, Helmholtzstra\ss e 18, 89081 Ulm, Germany.}
	\email{leonie.langer@uni-ulm.de}
	
	\begin{abstract}
		We derive an $H^{-1}$-gradient flow of the elastic energy which preserves the enclosed area of evolving planar curves. 
		For this new sixth-order evolution equation, we prove a global existence result. Additionally, by penalizing the length, we show convergence to an area constrained critical point of the elastic energy.
	\end{abstract}
	
	\maketitle

	
	\section{Introduction}
	\label{sec:intro}
	We model a closed planar elastic wire by a sufficiently smooth regular curve $\gamma\colon\sS\to\RR^2$.  
	The elastic energy of the wire is then given by
	\begin{align}
		\label{eq:ee}
		\sE(\gamma)=\frac12\int_{\sS}\vert\kappa\vert^2\intd s,
	\end{align}
	where $\kappa=\partial_s^2\gamma$ (with $\partial_s=\vert\partial_x\gamma\vert^{-1}\partial_x$) is the curvature vector and $\intd s=\vert\partial_x\gamma\vert\intd x$ denotes the arclength measure of $\gamma$.
	The study of the elastic energy \eqref{eq:ee} dates back to the 17th and 18th centuries. At that time, Jacob Bernoulli, his nephew Daniel, and Euler derived it to determine the shape of an elastic beam bent by an attached weight, see for example \cite{L08}. Even today, the elastic energy finds applications across various fields, from engineering and physics to biomechanics and materials science. It appears as the $\Gamma$-limit of an energy functional for lipid bilayer membranes (see \cite{PR09}) and of discrete bending energies for atomic chains (see \cite{ADG17}, \cite{EGW18}).
	
	If the elastic wire is extensible, the elastic energy can be made arbitrarily small by enlarging the wire. 
	In many situations, this is unphysical and thus also a penalization of length is taken into account.
	We consider
	\begin{align}
		\label{eq:pee}
		\sE_\lambda(\gamma)=\sE(\gamma)+\lambda\sL(\gamma)=\frac12\int_{\sS}\left(\vert\kappa\vert^2+2\lambda\right)\intd s
	\end{align}
	for some $\lambda\geq0$. When $\lambda>0$, the term $\lambda\sL$ measures the energy expended on stretching the extensible wire.
	
	Energetically favoured states of the elastic wire are critical points of \eqref{eq:pee}. 
	These critical points are called elasticae and satisfy the Euler--Lagrange equation
	\begin{align}
		\label{eq:ELelasticae}
		\nabla_s^2\kappa+\frac12\vert\kappa\vert^2\kappa-\lambda\kappa=0.
	\end{align}
	Here, $\nabla_s=P^\perp\partial_s$, where $P^\perp$ is the normal projection along $\gamma$ defined by $P^\perp\phi=\phi-\langle\phi,\partial_s\gamma\rangle\partial_s\gamma$ for vector fields $\phi:\sS\to\RR^2$. For $\lambda>0$, solutions to \eqref{eq:ELelasticae} have been classified in several works, see for instance \cite{LS1984}, \cite{DHMV2008}, or \cite[Lemma 5.4]{MR2021}. In particular, the only closed elasticae are multifold coverings of the circle and of the figure eight elastica. For $\lambda=0$, no critical points of \eqref{eq:pee} exist, see e.g.\ \cite[Section 2.2]{MW24}.
	
	The most straightforward dynamic approach to study \eqref{eq:pee}, is to consider its $L^2$-gradient flow given by 
	\begin{align}
		\label{eq:clasef}
		\partial_t^\perp\gamma=-\nabla_s^2\kappa-\frac12\vert\kappa\vert^2\kappa+\lambda\kappa.
	\end{align}
	Here, $\partial_t^\perp=P^\perp\partial_t$. To account for the parametrization invariance of \eqref{eq:pee}, it is a common approach to only prescribe the normal component of velocity.
	
	In the last decades, several authors have studied \eqref{eq:clasef}
	in different variants. First Polden (\cite{P1996}) and then Dziuk, Kuwert, and Schätzle (\cite{DKS2002}) considered the case of closed curves. Later, the problem was extended to non-closed curves (see for example \cite{NO2014},
	\cite{DPS16}, \cite{Lin}) and to networks (see for example \cite{DALP2019}, 
	\cite{GMP2020}). In $\RR^2$, curves can be described by an inclination angle function, which yields not a fourth but a second order flow equation. This flow has been first considered by Wen (\cite{W1993}) and later by \cite{LL2015}, \cite{LL2018}. For a more detailed overview, we refer to the survey article \cite{MR4277362}. 
	
	In several works, inextensibility of the wire is assumed, i.e.\ the length of the evolving curve is kept fixed (\cite{K1996}, \cite{DALP2014}, 
	\cite{DKS2002}, \cite{RS2020}). 
	In contrast, to the best of our knowledge, there is only one work (\cite{O2007}) in which the elastic flow under the constraint of fixed enclosed  
	area is studied. The area preserving scenario is physically relevant, for instance, to describe an interface between an inner domain and an outer domain filled with two kinds of incompressible viscous fluids. This relevance is further indicated by the existence of several works exploring area preserving variants for other flows of curves, see for example \cite{G86},  
	or the recent articles \cite{W23}, \cite{K24}. In \cite[Section 6]{MR4277362}, the authors explicitly point out the absence of studies of an area preserving elastic flow. 
	
	In the aforementioned article \cite{O2007}, Okabe derives a system of equations with a nonlocal Lagrange multiplier that governs 
	the $L^2$-gradient of \eqref{eq:ee} under the constraints of inextensiblity and preserved enclosed area. 
	For this nonlocal system, existence of a unique solution and convergence is proven. 
	In some models, local evolution equations that do not depend on any action at a distance are preferred (see for example \cite{F2000}).
	This motivates to search for an evolution equation which decreases \eqref{eq:pee} and preserves the enclosed area without the need for nonlocal Lagrange multipliers. 
	Inspired by the Cahn--Hilliard equation, which can be seen as an $H^{-1}$-gradient flow preserving a fixed total mass without any nonlocal Lagrange multiplier, we consider a $H^{-1}$-gradient approach for \eqref{eq:pee}. To do so, we first define appropriate Hilbert spaces that depend on the curve $\gamma$ and, if the curve evolves, on time. The definition includes a zero-integral condition, which allows for a Poincaré-type inequality, as also found in the derivation of the Cahn-Hilliard equation (see e.g.\ \cite{F2000}). 
	In these Hilbert spaces, we compute the gradient of \eqref{eq:pee} and take its negative as the flow velocity, see \Cref{sec:deriv}. This approach transforms the classical elastic flow \eqref{eq:clasef} with fourth order velocity into the sixth order flow
	\begin{align}
		\label{eq:fleq}
		\partial_t^\perp\gamma=\nabla_s^2\Big(\nabla_s^2\kappa+\frac12\vert\kappa\vert^2\kappa-\lambda\kappa\Big).
	\end{align}
	The gradient flow structure ensures that under the flow \eqref{eq:fleq}, the energy \eqref{eq:pee} decreases (see \Cref{lem:decE}). Moreover, the signed enclosed area $\sA$ is preserved, i.e.
	\begin{align}
		\label{eq:A}
		\sA(\gamma)=-\frac12\int_\sS\langle\gamma,\nu\rangle\intd s,
	\end{align}
	is constant (see \Cref{lem:presA}).
	Here, $\nu$ denotes the normal vector given by the counter-clockwise rotation of $\partial_s\gamma$ by $\pi/2$. 
	Hence, \eqref{eq:fleq} is a flow for the elastic energy that preserves the enclosed area without nonlocality. 
	In this article we aim to understand this flow.
	\pagebreak
	
	We investigate a smooth regular initial datum $\gamma_0\colon\sS\to\RR^2$ with rotation index $\omega_0\in\ZZ$ and enclosed area $\sA(\gamma_0)=A_0\in\RR$ evolving under the area preserving elastic flow \eqref{eq:fleq}. So we study the initial value problem 
	\begin{align}\label{eq:ivp}
		\begin{cases}
			\partial_t^\perp\gamma=\nabla_s^2\big(\nabla_s^2\kappa+\frac12\vert\kappa\vert^2\kappa-\lambda\kappa\big)\quad&\text{ in }(0,T)\times\sS,\\
			\gamma(0,\cdot)=\gamma_0&\text{ on }\sS.
		\end{cases}
	\end{align}
	Throughout the entire article, the enclosed area $A_0\in\RR$ of the initial datum is considered arbitrary, including the case where $A_0=0$. We also consider arbitrary rotation index $\omega_0\in\ZZ$, in particular allowing the case $\omega_0=0$.
	
	Even though our approach increases the order of the classical elastic flow by two,
	the behavior of solutions does not deteriorate. 
	We show long time existence of solutions and, if we penalize the length,
	convergence. By prescribing only the normal part of the velocity, we achieve uniqueness of the solution up to reparametrization. We refer to this as geometric uniqueness.
	
	\begin{thm}[Global existence and convergence]\label{thm:combined}
		For any $\lambda\geq0$, there exists a geometrically unique global smooth solution $\gamma\colon[0,\infty)\times\sS\to\RR^2$ of \eqref{eq:ivp}.
		If $\lambda>0$, there exists a family of smooth diffeomorphisms $\Phi(t)\colon\sS\to\sS$, $t\in(0,\infty)$, such that $\gamma(t,\Phi(t,\cdot))$ converges smoothly for $t\to\infty$ to a stationary solution $\gamma_\infty$.
	\end{thm}
	
	By a stationary solution, we mean a smooth regular curve $\gamma\colon\sS\to\RR^2$ with rotation index $\omega_0$ and $\sA(\gamma)=A_0$ such that 
	$
	\nabla_s^2\big(\nabla_s^2\kappa+\tfrac12\vert\kappa\vert^2\kappa-\lambda\kappa\big)=0.
	$
	Since $\gamma$ describes a closed curve, this is equivalent to 
	\begin{align}
		\label{eq:statsolvec2}
		\nabla_s\big(\nabla_s^2\kappa+\tfrac12\vert\kappa\vert^2\kappa-\lambda\kappa\big)=0.
	\end{align}
	This equation is the classical elasticae equation \eqref{eq:ELelasticae} with a possibly nonzero constant multiple of the normal vector on the right hand side.
	The presence of this nonzero term implies that there are stationary solutions which do not correspond to a multiple covering of a circle or of the figure eight elastica.
	The existence of non-elastica solutions of \eqref{eq:statsolvec2} was proven by deriving an explicit parametrization using Jacobi elliptic functions, see for example \cite{VDM08}.
	By showing convergence of the flow, we provide an
	alternative proof for the existence of solutions to \eqref{eq:statsolvec2} that do not satisfy \eqref{eq:ELelasticae}, meaning they are not elasticae.
	
	\Cref{thm:combined} implies in particular that the solution remains within a compact set for all times if $\lambda>0$. Since no maximum principle is available for higher-order flows like \eqref{eq:fleq}, which would allow us to conclude this property by comparison arguments, this is a remarkable result. The same holds true for the classical elastic flow.
	
	For the convergence result in \Cref{thm:combined}, the length penalization is necessary. Indeed, for $\lambda=0$, solutions do not converge in general. If $\lambda=0$ and $A_0=0$, the length of the evolving curve always becomes unbounded for $t\to\infty$ (see \Cref{lem:exnonconv}). 
	
	To prove the global in time existence in \Cref{thm:combined}, we first establish suitable interpolation inequalities, which we then use to show that the solution remains bounded in finite time. For the proof of convergence, we use that stationary solutions are area constrained critical points of \eqref{eq:pee} and rely on a suitable version of a constrained \L ojasiewicz--Simon gradient inequality, see \cite{ConstrLoja}. The $H^{-1}$-gradient flow structure obstructs the usual procedure to show convergence (see for example \cite{MantegazzaPozzetta21}, \cite{DPS16} or \cite{RS2020}). Indeed, the time-dependent Hilbert spaces in which the energy decreases do not coincide as sets, unlike for the $L^2$-gradient flow. 
	By considering suitably defined time-independent norms, we overcome this difficulty. 
	
	We point out that in \cite{AMCWW2020}, an evolution equation similar to \eqref{eq:fleq} is studied. The authors consider the $L^2$-gradient flow of the Dirichlet energy of the scalar curvature (without penalization of length), which also yields a sixth order equation. However, this equation does not preserve the enclosed area in general and shows a completely different asymptotic behavior.
	In \cite{AMCWW2020}, the authors classify the stationary solutions as $\omega$-fold covered circles. In particular, this shows that the solution for an initial datum with $\omega=0$ cannot converge. Convergence to an $\omega$-fold covered circle is shown under the assumption that the length of the evolving curve remains bounded, which is the case for small initial energy. The length preserving and area preserving variants of the flow discussed in \cite{AMCWW2020} (with nonlocal Lagrange multipliers), are studied in \cite{MW22} and \cite{W21}.
	
	This article is structured as follows. In \Cref{sec:deriv} we define an appropriate Hilbert space setting in which the area preserving elastic flow is derived as a gradient flow.  \Cref{sec:firstprop} lists some basic properties of the evolution equation; in particular, we see that the flow \eqref{eq:fleq} indeed preserves the enclosed area. \Cref{sec:shorttime} briefly addresses the question of short time existence, before we prove global in time existence and subconvergence (for $\lambda>0$) in \Cref{sec:globalex}. \Cref{sec:conv} is devoted to prove full convergence using a suitable \L ojasiewicz--Simon gradient inequality, which we first derive.
	
	\section{Derivation of the flow equation }
	\label{sec:deriv}
	
	\subsection{An appropriate Hilbert space setting }
	\label{subsec:Hgamma}
	
	For a given smooth regular curve $\gamma\colon\sS\to\RR^2$, we define the vector space
	\begin{align}
		H_\gamma:=\Big\lbrace Y\colon\sS\to\RR^2: Y=\langle Y,\nu_\gamma\rangle\nu_\gamma,\; Y\in L^2(\intd s_\gamma),\; \nabla_{s_\gamma}Y\in L^2(\intd s_\gamma),\; \int_\sS\langle\nu_\gamma,Y\rangle\intd s_\gamma=0\Big\rbrace.
	\end{align}
	Here, 
	$\intd s_\gamma$ denotes the arclength measure of $\gamma$, $\nabla_{s_\gamma}$ is the normal projection of the arclength derivative $\partial_{s_\gamma}$ along $\gamma$ and $\nu_\gamma$ is the normal vector of $\gamma$.
	On $H_\gamma$, we consider the inner product
	\begin{align}
		\langle Y_1,Y_2\rangle_{H_\gamma}=\langle \nabla_{s_\gamma}Y_1,\nabla_{s_\gamma}Y_2\rangle_{L^2(\intd s_\gamma)}
	\end{align} 
	which induces the norm 
	\begin{align}\label{eq:normHgamma}
		\Vert Y\Vert_{H_\gamma}=\Vert\nabla_{s_\gamma} Y\Vert_{L^2(\intd s_\gamma)}.
	\end{align}
	
	This is justified by the following lemma, which also shows that on $H_\gamma$, $\Vert\cdot\Vert_{H_\gamma}$ is equivalent to the Sobolev norm $\Vert\cdot\Vert_{W^{1,2}}=\Vert\cdot\Vert_{W^{1,2}(\intd x)}$.
	
	\begin{lem}\label{lem:normequiv}
		There exist constants $c_1,c_2>0$, depending on
		$\sL(\gamma)$, $\Vert\kappa_\gamma\Vert_{C^0}$, $\Vert\partial_x\gamma\Vert_{C^0}$,
		and $\min_\sS\vert\partial_x\gamma\vert$,
		such that 
		\begin{align} \label{eq:normequiv0}
			c_1\Vert Y\Vert_{W^{1,2}}\leq\Vert Y\Vert_{H_\gamma}\leq c_2\Vert Y\Vert_{W^{1,2}}\quad\text{ for }Y\in H_\gamma.
		\end{align}
		In particular, $(H_\gamma, \Vert\cdot\Vert_{H_\gamma})$ is a Hilbert space.
	\end{lem}
	
	\begin{proof}
		Let $y:=\langle\nu_\gamma,Y\rangle$. Then $\partial_{s_\gamma}y\in L^2(\intd s_\gamma)$ and
		$
		\int_\sS y\intd s_\gamma=0.
		$
		Thus, with Poincaré's inequality, we obtain
		\begin{align} \label{eq:normequiv1}
			\Vert Y\Vert_{L^2(\intd {s_\gamma})}^2
			=\int_\sS y^2\intd s_\gamma
			\leq C\int_\sS (\partial_{s_\gamma}y)^2\intd s_\gamma
			=C\int_\sS\langle\nabla_{s_\gamma}Y,\nu_\gamma\rangle^2\intd s_\gamma
			=C\Vert\nabla_{s_\gamma}Y\Vert_{L^2(\intd {s_\gamma})}^2
		\end{align}
		for a constant $C$ depending on $\sL(\gamma)$. 
		Moreover, $Y=\langle Y,\nu_\gamma\rangle\nu_\gamma$ implies that 
		\begin{align}\label{eq:partialsY}
			\partial_{s_\gamma}Y
			=\nabla_{s_\gamma}Y-
			\langle Y,\nu_\gamma\rangle\langle\kappa_\gamma,\nu_\gamma\rangle\partial_{s_\gamma}\gamma.
		\end{align}
		Here, we use $\partial_{s_\gamma}\nu_\gamma=-\langle\kappa_\gamma,\nu_\gamma\rangle\partial_{s_\gamma}\gamma$. 
		Equation \eqref{eq:partialsY} yields together with \eqref{eq:normequiv1} that
		\begin{align} \label{eq:normequiv2}
			C_1 \Vert\nabla_{s_\gamma}Y\Vert_{L^2(\intd {s_\gamma})}^2
			\leq\Vert Y\Vert_{L^2(\intd {s_\gamma})}^2+\Vert\partial_{s_\gamma}Y\Vert_{L^2(\intd {s_\gamma})}^2
			\leq C_2 \Vert\nabla_{s_\gamma}Y\Vert_{L^2(\intd {s_\gamma})}^2
		\end{align}
		for constants
		$C_1>0$ depending on the inverse of $\max_\sS \vert\kappa_\gamma\vert^2$ and $C_2>0$ depending on $\sL(\gamma)$ and $\max_\sS \vert\kappa_\gamma\vert^2$.
		From $\partial_{ {s_\gamma}}Y=\vert\partial_x\gamma\vert^{-1}\partial_xY$ and $\intd s_\gamma=\vert\partial_x\gamma\vert\intd x$, we conclude
		\begin{align}
			\Vert Y\Vert^2_{W^{1,2}}=\Vert Y\Vert^2_{L^2}+\Vert\partial_xY\Vert^2_{L^2}
			\leq C\big(\Vert Y\Vert^2_{L^2(\intd s_{\gamma})}+\Vert\partial_{s_\gamma}Y\Vert^2_{L^2(\intd s_\gamma)}\big)
		\end{align}
		for some $C$ depending on $\min_\sS\vert\partial_x\gamma\vert$ and $\max_\sS\vert\partial_x\gamma\vert$. From this, the left inequality of \eqref{eq:normequiv0} follows with \eqref{eq:normequiv1} and \eqref{eq:normequiv2}. The right inequality follows similarly leading to the equivalence of the norms.  
		Moreover, $H_\gamma$ is a closed subspace of $W^{1,2}(\sS;\RR^2)$. Thus, $(H_\gamma, \Vert\cdot\Vert_{H_\gamma})$ is a Hilbert space.
	\end{proof}
	
	The Hilbert space $H_\gamma$ is geometric in the following sense. 
	
	\begin{lem}
		Let $\tilde\gamma=\gamma\circ\varphi\colon\sS\to\RR^2$ be a smooth reparametrization of $\gamma$. If $Y\in H_\gamma$, then $Y\circ\varphi\in H_{\tilde\gamma}$ and $\Vert Y\Vert_{H_\gamma}=\Vert Y\circ\varphi\Vert_{H_{\tilde\gamma}}$. For $Y_1,Y_2\in H_\gamma$, we have $\langle Y_1\circ\varphi,Y_2\circ\varphi\rangle_{H_{\tilde\gamma}}=\langle Y_1,Y_2\rangle_{H_\gamma}$.
	\end{lem}
	
	\begin{proof}
		Let $Y\in H_\gamma$. If the smooth diffeomorphism $\varphi$ of the reparametrization is orientation preserving, then $\nu_{\tilde\gamma}=\nu_\gamma\circ\varphi$. If $\varphi$ changes the orientation of $\gamma$, then $\nu_{\tilde\gamma}=-\nu_\gamma\circ\varphi$. 
		In both cases, we obtain
		\begin{align}
			\tilde Y:=
			Y\circ\varphi
			=\langle Y\circ\varphi,\nu_\gamma\circ\varphi\rangle\,\nu_\gamma\circ\varphi
			=\langle\tilde Y,\nu_{\tilde\gamma}\rangle\,\nu_{\tilde\gamma}.
		\end{align}
		Moreover, substitution and $\vert\partial_x\tilde\gamma(x)\vert=\vert\varphi'(x)\,\partial_x\gamma(\varphi(x))\vert$ yields
		\begin{align}
			\Vert \tilde Y\Vert^2_{L^2(\intd s_{\tilde\gamma})}
			=\int_\sS\vert Y(\varphi(x))\vert^2\vert\partial_x\tilde\gamma(x)\vert\intd x
			=\int_\sS\vert Y\vert^2\intd s_\gamma<\infty
		\end{align}
		as well as
		\begin{align} \label{eq:normHtildegamma}
			&\Vert\nabla_{s_{\tilde\gamma}}\tilde Y\Vert_{L^2(\intd s_{\tilde\gamma})}
			=\int_\sS(\langle\partial_{s_{\tilde\gamma}}\tilde Y,\nu_{\tilde\gamma}\rangle)^2\intd s_{\tilde\gamma}
			=\int_\sS\frac{1}{\vert\partial_x\tilde\gamma(x)\vert}\big(\langle\partial_x(Y(\varphi(x))),\nu_\gamma(\varphi(x))\rangle\big)^2\intd x\\
			&\quad=\int_\sS\frac{1}{\vert\partial_x\gamma(z)\vert}(\langle\partial_xY(z),\nu_\gamma(z)\rangle)^2\intd z
			=\int_\sS(\langle\partial_{s_\gamma}Y,\nu_\gamma\rangle)^2\intd s_\gamma
			=\Vert\nabla_{s_{\gamma}} Y\Vert_{L^2(\intd s_{\gamma})}<\infty.
		\end{align}
		Along the same lines, we obtain
		\begin{align}
			\int_\sS\langle\nu_{\tilde\gamma},\tilde Y\rangle\intd s_{\tilde\gamma}
			=\int_\sS\langle\nu_\gamma(\varphi(x)),Y(\varphi(x))\rangle\vert\partial_x\tilde\gamma(x)\vert\intd x
			=0,
		\end{align}
		since $\int_\sS\langle\nu_\gamma,Y\rangle\intd s_\gamma=0$.
		Everything together results in $\tilde Y\in H_{\tilde\gamma}$. Moreover, with \eqref{eq:normHgamma}, \eqref{eq:normHtildegamma} yields $\Vert Y\Vert_{H_\gamma}=\Vert \tilde Y\Vert_{H_{\tilde\gamma}}$.
		The rest of the statement follows similarly.
	\end{proof}
	
	In the following, we denote by $H_\gamma^{-1}$ the dual space of $H_\gamma$. 
	
	\begin{defn}
		Let $\phi\in H_\gamma^{-1}$. We call $X\in H_\gamma$ a \textit{weak solution} of 
		\begin{align}\label{eq:weaksol}
			-\nabla_{s_\gamma}^2X=\phi
		\end{align}
		if $\phi(Y)=\langle X,Y\rangle_{H_\gamma}=\langle\nabla_{s_\gamma}X,\nabla_{s_\gamma}Y\rangle_{L^2(\intd s_\gamma)}$ for all $Y\in H_\gamma$.
	\end{defn}
	
	Note that by the Riesz representation theorem, for all $\phi\in H_\gamma^{-1}$, there exists a unique weak solution 
	of \eqref{eq:weaksol}.\medskip
	
	We define an inner product on $H_\gamma^{-1}$ by 
	\begin{align}
		\langle\phi_1,\phi_2\rangle_{H_\gamma^{-1}}:=\langle X_1,X_2\rangle_{H_\gamma}=\langle\nabla_{s_\gamma}X_1,\nabla_{s_\gamma}X_2\rangle_{L^2(\intd s_\gamma)}, 
	\end{align}
	where $X_i\in H_\gamma$ is the unique weak solution of $-\nabla_{s_\gamma}^2X_i=\phi_i$, $i=1,2$. Moreover, we have
	\begin{align}\label{eq:normHgamma-1konkret}
		\Vert X\Vert_{H_\gamma}=\Vert\phi\Vert_{H_\gamma^{-1}}.
	\end{align}
	
	\begin{rem}
		If the weak solution $X$ of \eqref{eq:weaksol} is sufficiently smooth, precisely $\nabla_{s_\gamma}^2 X\in L^2(\intd s_\gamma)$, then $\phi\in H_\gamma^{-1}$ is represented by $\phi=\langle-\nabla_{s_\gamma}^2X,\cdot\rangle_{L^2(\intd s_\gamma)}$. We identify the functional $\phi$ with the function $-\nabla_{s_\gamma}^2X$ using this $L^2$-pairing. With this identification, $H_\gamma\subseteq L^2(\intd s_\gamma)\subseteq H^{-1}_\gamma$. Moreover, this identification allows us to express the $H_\gamma^{-1}$-norm as
		\begin{align}\label{eq:normHgamma-1}
			\Vert\phi\Vert_{H_\gamma^{-1}}=\sup_{Y\in H_\gamma\setminus\lbrace0\rbrace}\frac{1}{\Vert Y\Vert_{H_\gamma}}\int_\sS\langle\phi,Y\rangle\intd s_\gamma.
		\end{align}
	\end{rem}
	
	\begin{rem}
		Let $\phi\in C^\infty(\sS;\RR^2)\cap H_\gamma$. Then $\phi$ is interpreted as an element of $H_\gamma^{-1}$ as $H_\gamma\ni Y\mapsto\langle\phi,Y\rangle_{L^2(\intd s_\gamma)}$. In this case, the weak solution $X$ of \eqref{eq:weaksol} is smooth (and coincides with the strong solution). Indeed, with \eqref{eq:partialsY}, for all $Y\in W^{1,2}(\intd x)$, we have
		\begin{align}
			&C(\gamma)\left\vert\langle\partial_xX,\partial_xY\rangle_{L^2(\intd x)}\right\vert
			\leq \left\vert\langle\partial_{s_\gamma}X,\partial_{s_\gamma}Y\rangle_{L^2(\intd s_\gamma)}\right\vert\\
			&\quad\leq \left\vert\langle\nabla_{s_\gamma}X,\nabla_{s_\gamma}Y\rangle_{L^2(\intd s_\gamma)}\right\vert
			+\left\vert\int_\sS\vert\kappa_\gamma\vert\langle X,\nu_\gamma\rangle\langle\partial_{s_\gamma}\gamma,\partial_{s_\gamma}Y\rangle\intd s_\gamma\right\vert\\
			&\quad\leq \vert\phi(Y)\vert 
			+\left\vert\int_\sS(\partial_{s_\gamma}\vert\kappa_\gamma\vert)\langle X,\nu_\gamma\rangle\langle\partial_{s_\gamma}\gamma,Y\rangle\intd s_\gamma\right\vert
			+\left\vert\int_\sS\vert\kappa_\gamma\vert\langle \partial_{s_\gamma}X,\nu_\gamma\rangle\langle\partial_{s_\gamma}\gamma,Y\rangle\intd s_\gamma\right\vert\\
			&\qquad+\left\vert\int_\sS\vert\kappa_\gamma\vert\langle X,\partial_{s_\gamma}\nu_\gamma\rangle\langle\partial_{s_\gamma}\gamma,Y\rangle\intd s_\gamma\right\vert
			+\left\vert\int_\sS\vert\kappa_\gamma\vert\langle X,\nu_\gamma\rangle\langle\kappa_\gamma,Y\rangle\intd s_\gamma\right\vert\\
			&\quad\leq C\Vert Y\Vert_{L^2(\intd s_\gamma)}
			\leq C\Vert Y\Vert_{L^2(\intd x)},
		\end{align}
		for some $C=C(\gamma,\Vert\phi\Vert_{L^2(\intd s_\gamma)}, \Vert X\Vert_{H_\gamma})$, that changes along the lines. It follows that $X\in W^{2,2}(\intd x)\subset C^1(\sS;\RR^2)$. With similar arguments we achieve $X\in C^\infty(\sS;\RR^2)$.
	\end{rem}
	
	On $H_\gamma$, we have the following interpolation inequality. 
	
	\begin{lem}\label{lem:interpol2}
		Let $\phi\in C^\infty(\sS;\RR^2)\cap H_\gamma$. Then
		\begin{align}
			\Vert\phi\Vert_{L^2(\intd s_\gamma)}^2\leq\Vert\phi\Vert_{H_\gamma^{-1}}\Vert\phi\Vert_{H_\gamma}.
		\end{align}
	\end{lem}
	
	\begin{proof}
		Let $X\in C^\infty(\sS;\RR^2)\cap H_\gamma$ be the solution of $-\nabla_{s_\gamma}^2X=\phi$. With \eqref{eq:normHgamma-1konkret}, we have
		\begin{align}
			\Vert\phi\Vert^2_{L^2(\intd s_{\gamma})}
			&=-\int_\sS\langle\nabla_{s_\gamma}^2X,\phi\rangle\intd s_\gamma
			=\int_\sS\langle\nabla_{s_\gamma}X,\nabla_{s_\gamma}\phi\rangle\intd s_\gamma\\
			&\leq\Vert\nabla_{s_\gamma}X\Vert_{L^2(\intd s_\gamma)}\Vert\nabla_{s_\gamma}\phi\Vert_{L^2(\intd s_\gamma)}
			=\Vert\phi\Vert_{H_\gamma^{-1}}\Vert\phi\Vert_{H_\gamma}. \qedhere
		\end{align}
	\end{proof}
	
	Moreover, on $H_\gamma$, the $H_\gamma^{-1}$-norm is equivalent to another norm, which we will use in \Cref{sec:conv} as an auxiliary norm.
	
	\begin{lem}\label{lem:auxnorm}
		Let $\phi\in C^\infty(\sS;\RR^2)\cap H_\gamma$. Define
		\begin{align}
			\Vert\phi\Vert_{\star}:=
			\sup_{Y\in W^{1,2}(\sS;\RR^2)\setminus\lbrace0\rbrace}\;
			\frac{1}{\Vert Y\Vert_{W^{1,2}}}\int_\sS\langle\phi,Y\rangle\intd x.
		\end{align}
		Then there exist constants $c_1,c_2>0$ depending on $\Vert\kappa_\gamma\Vert_{C^0}$, $\sL(\gamma)$ as well as $\min_\sS\vert\partial_x\gamma\vert$ and $\Vert\partial_x\gamma\Vert_{C^0}$ such that 
		\begin{align}
			c_1\Vert\phi\Vert_{\star}\leq\Vert\phi\Vert_{H_\gamma^{-1}}\leq c_2\Vert\phi\Vert_{\star}. 
		\end{align}
	\end{lem}
	
	\begin{proof}
		For the proof, we consider
		\begin{align}
			\Vert\phi\Vert_{\textup{aux}}:=
			\sup_{Y\in W^{1,2}(\sS;\RR^2)\setminus\lbrace0\rbrace}\;
			\frac{1}{\Vert Y\Vert_{W^{1,2}}}\int_\sS\langle\phi,Y\rangle\intd s_\gamma
		\end{align}
		and show the equivalence of $\Vert\cdot\Vert_{\textup{aux}}$ and $\Vert\cdot\Vert_{H_\gamma^{-1}}$ on $C^\infty(\sS;\RR^2)\cap H_\gamma$. Estimating the arclength element $\vert\partial_x\gamma\vert$ from below and from above, this equivalence yields the claim.
		
		For $Y\in H_\gamma\subset W^{1,2}(\sS;\RR)$, \Cref{lem:normequiv} yields the existence of a constant $c_2$ depending on $\gamma$ such that 
		\begin{align}
			\Vert\phi\Vert_{H_\gamma^{-1}}
			&=\sup_{Y\in H_\gamma\setminus\lbrace0\rbrace}\;
			\frac{1}{\Vert Y\Vert_{H_\gamma}}\int_\sS\langle\phi,Y\rangle\intd s_\gamma
			\leq c_2\sup_{Y\in H_\gamma\setminus\lbrace0\rbrace}\;
			\frac{1}{\Vert Y\Vert_{W^{1,2}}}\int_\sS\langle\phi,Y\rangle\intd s_\gamma
			\leq c_2\Vert\phi\Vert_{\textup{aux}}.
		\end{align}
		For the other inequality, we first note that if $\phi\in H_\gamma$, then $\phi$ is orthogonal to $\partial_s\gamma$ and thus $\langle\phi,Y\rangle=\langle\phi,Y^\perp\rangle$, where $Y^\perp:=\langle Y,\nu_\gamma\rangle\nu_\gamma$ for any $Y\in W^{1,2}(\sS;\RR^2)$.
		Moreover, $\vert Y^\perp\vert^2\leq\vert Y\vert^2$ and
		\begin{align}
			\vert \partial_{s_\gamma} (Y^\perp)\vert^2
			&=\vert\partial_{s_\gamma} Y-\langle\partial_{s_\gamma} Y,\partial_s\gamma\rangle\partial_s\gamma-\langle Y,\kappa_\gamma\rangle\partial_s\gamma-\langle Y,\partial_s\gamma\rangle\kappa_\gamma\vert^2\\
			&=\vert\partial_{s_\gamma}Y\vert^2-2(\langle\partial_{s_\gamma}Y,\partial_s\gamma\rangle)^2-2\langle Y,\partial_s\gamma\rangle\langle\partial_{s_\gamma}Y,\kappa_\gamma \rangle
			+(\langle\partial_{s_\gamma}Y,\partial_s\gamma\rangle)^2\\
			&\quad+(\langle Y,\kappa_\gamma\rangle)^2+(\langle Y,\partial_s\gamma\rangle)^2\vert\kappa_\gamma\vert^2\\
			&=\vert\partial_{s_\gamma}Y\vert^2-(\langle\partial_{s_\gamma}Y,\partial_s\gamma\rangle)^2-2\langle Y,\partial_s\gamma\rangle\langle\partial_{s_\gamma}Y,\kappa_\gamma\rangle+\vert\kappa_\gamma\vert^2\vert Y\vert^2\\
			&\leq C(\vert Y\vert^2+\vert\partial_{s_\gamma}Y\vert^2)
			\label{eq:komanna}
		\end{align}
		for a constant $C$ depending on bounds on $\kappa_\gamma$. Thus, $\Vert Y^\perp\Vert_{W^{1,2}}\leq C\Vert Y\Vert_{W^{1,2}}$ for a constant $C$ depending on bounds on $\kappa_\gamma$ and the arclength element $\vert\partial_x\gamma\vert$. As $\langle\phi,Y\rangle=\langle\phi,Y^\perp\rangle$, we conclude that 
		\begin{align}
			\Vert\phi\Vert_{\textup{aux}}
			=\sup_{Y\in W^{1,2}\setminus\lbrace0\rbrace}\;
			\frac{1}{\Vert Y\Vert_{W^{1,2}}}\int_\sS\langle\phi,Y\rangle\intd s_\gamma
			\leq
			C\sup_{\substack{Y\in W^{1,2}\setminus\lbrace0\rbrace,\\ Y=\langle Y,\nu_\gamma\rangle\nu_\gamma}}\;\frac{1}{\Vert Y\Vert_{W^{1,2}}}\int_\sS\langle\phi,Y\rangle\intd s_\gamma
			\label{eq:smiley}
		\end{align}
		with $C$ depending on the same quantities as above.
		Now we note that $\phi\in H_\gamma$ also implies that 
		\begin{align}
			\int_\sS\langle\phi,Y\rangle\intd s_\gamma
			=\int_\sS\langle\phi,Y-\frac{\nu_\gamma}{\sL(\gamma)}\int_\sS\langle Y,\nu_\gamma\rangle\intd s_\gamma\rangle\intd s_\gamma
			=:\int_\sS\langle \phi,Y-C_\gamma(Y)\nu_\gamma\rangle\intd s_\gamma
		\end{align}
		for $Y\in W^{1,2}(\sS;\RR^2)$. Moreover, for $Y\in W^{1,2}(\sS;\RR^2)\cap\lbrace Y=\langle Y,\nu_\gamma\rangle\nu_\gamma\rbrace$, we have
		\begin{align}\label{eq:tk1}
			\int_\sS\vert Y-C_\gamma(Y)\nu_\gamma\vert^2\intd s_\gamma
			=\int_\sS\vert Y\vert^2\intd s_\gamma-\sL(\gamma)C_\gamma(Y)^2\leq\int_\sS\vert Y\vert^2\intd s_\gamma
		\end{align}
		and 
		\begin{align}
			\int_\sS\vert\partial_{s_\gamma}(Y-C_\gamma(Y)\nu_\gamma)\vert^2\intd s_\gamma
			&=\int_\sS\vert\partial_{s_\gamma}Y\vert^2\intd s_\gamma+2C_\gamma(Y)\int_\sS\langle\partial_{s_\gamma}Y,\partial_s\gamma\rangle\langle\kappa_\gamma,\nu_\gamma\rangle\intd s_\gamma\\
			&\quad+C_\gamma(Y)^2\int_\sS\vert\kappa_\gamma\vert^2\intd s_\gamma.\label{eq:tk2}
		\end{align}
		Since 
		\begin{align}
			C_\gamma(Y)^2\leq\frac{1}{\sL(\gamma)^2}\Vert Y\Vert^2_{L^1(\intd s_\gamma)}\leq\frac{1}{\sL(\gamma)}\Vert Y\Vert_{L^2(\intd s_\gamma)}^2,
		\end{align}
		\eqref{eq:tk1} and \eqref{eq:tk2} imply that 
		$
		\Vert Y-C_\gamma(Y)\nu_\gamma\Vert_{W^{1,2}}\leq C\Vert Y\Vert_{W^{1,2}}
		$
		for a constant $C$ depending on bounds on $\kappa_\gamma $ and $\vert\partial_x\gamma\vert$ and the energy $\sE(\gamma)$. We conclude from \eqref{eq:smiley} that 
		\begin{align}
			\Vert\phi\Vert_{\textup{aux}}\leq
			C\!\sup_{\substack{Y\in W^{1,2}\setminus\lbrace 0\rbrace,\\ Y=\langle Y,\nu_\gamma\rangle\nu_\gamma}}\;\frac{1}{\Vert Y\Vert_{W^{1,2}}}\!\int_\sS\langle\phi,Y\rangle\intd s_\gamma
			\leq C\!\sup_{\substack{Y\in W^{1,2}\setminus\lbrace0\rbrace, Y=\langle Y,\nu_\gamma\rangle\nu_\gamma,\\ \int_\sS\langle Y,\nu_\gamma\rangle\intd s_\gamma=0}}\;\frac{1}{\Vert Y\Vert_{W^{1,2}}}\!\int_\sS\langle\phi,Y\rangle\intd s_\gamma,
		\end{align}
		with $C$ depending on the quantities listed above. By the definition of $H_\gamma$ and \Cref{lem:normequiv}, this is
		$
		\Vert\phi\Vert_{\textup{aux}}\leq
		C\Vert\phi\Vert_{H^{-1}_\gamma}.
		$
	\end{proof}
	
	\subsection{The area preserving elastic flow as a gradient flow }
	\label{subsec:H-1flow}
	
	Flow equation \eqref{eq:fleq} can be derived as the $H_\gamma^{-1}$-gradient flow. For this, we consider a smooth regular curve $\gamma\colon\sS\to\RR^2$ and a smooth normal variation $\phi\colon\sS\to\RR^2$ 
	such that $\phi\in H_\gamma\subseteq H_\gamma^{-1}$. For $\varepsilon>0$ small enough, the perturbed curve $\gamma_\varepsilon:=\gamma+\varepsilon\phi$ is also regular. 
	With $\nabla_{L^2}\sE_\lambda(\gamma)=\nabla_{s_\gamma}^2\kappa_\gamma+\frac12\vert\kappa_\gamma\vert^2\kappa_\gamma-\lambda\kappa_\gamma$, we have
	\begin{align}\label{eq:grad1}
		\sE_\lambda'(\gamma)(\phi)
		&=\langle\nabla_{L^2}\sE_\lambda(\gamma),\phi\rangle_{L^2(\intd s_\gamma)}.
	\end{align}
	In general, $\nabla_{L^2}\sE_\lambda(\gamma)\not\in H_\gamma$. So we define
	\begin{align}
		W:=\nabla_{L^2}\sE_\lambda(\gamma)-C_\gamma\nu_\gamma\in H_\gamma\cap C^\infty(\sS;\RR^2)
		\quad\text{with}\quad 
		C_\gamma:=\frac{1}{\sL(\gamma)}\int_\sS\langle\nu_\gamma,\nabla_{L^2}\sE_\lambda(\gamma)\rangle\intd s_\gamma.
	\end{align}
	With $X\in H_\gamma\cap C^\infty(\sS;\RR^2)$ as the weak solution of $-\nabla_{s_\gamma}^2X=\phi$ and $w=-\nabla_{s_\gamma}^2W$, we obtain
	\begin{align}
		\langle w,\phi\rangle_{H_\gamma^{-1}}
		&=\langle\nabla_{s_\gamma}W,\nabla_{s_\gamma}X\rangle_{L^2(\intd s_\gamma)}
		=\int_\sS\langle W,-\nabla_{s_\gamma}^2X\rangle\intd s_\gamma
		=\int_\sS\langle \nabla_{L^2}\sE_\lambda(\gamma),\phi\rangle\intd s_\gamma
		\\ &=\langle\nabla_{L^2}\sE_\lambda(\gamma),\phi\rangle_{L^2(\intd s_\gamma)}.\label{eq:grad2}
	\end{align}
	Here, we used in the third step that 
	\begin{align}
		\int_\sS\langle C_\gamma\nu_\gamma,\nabla_{s_\gamma}^2X\rangle\intd s_\gamma
		=0,
	\end{align}
	since $\nabla_{s_\gamma}^2X=-\phi\in H_\gamma$.
	Comparing \eqref{eq:grad1} and \eqref{eq:grad2}, we define the $H_\gamma^{-1}$-gradient 
	\begin{align}
		\nabla_{H_\gamma^{-1}}\sE_\lambda(\gamma):=w
		=-\nabla_{s_\gamma}^2\big(\nabla_{L^2}\sE_\lambda(\gamma)-C_\gamma\nu_\gamma\big)
		=-\nabla_{s_\gamma}^2\big(\nabla_{s_\gamma}^2\kappa_\gamma+\tfrac12\vert\kappa_\gamma\vert^2\kappa_\gamma-\lambda\kappa_\gamma\big).
	\end{align}
	To properly adress the geometric nature of the problem, we set only the normal component of velocity equal to the negative gradient, yielding 
	\begin{align}
		\partial_t^\perp \gamma=-\nabla_{H_\gamma^{-1}}\sE_\lambda(\gamma),
	\end{align}
	which is \eqref{eq:fleq}. Note that, in general, the space with respect to which the gradient is computed changes in time.
	
	\section{Basic properties of the flow }
	\label{sec:firstprop}
	
	In order to incorporate the invariance of the elastic energy under reparametrizations, we defined \eqref{eq:ivp} in such a way that a reparametrization of a solution again solves \eqref{eq:ivp} with reparametrized initial data. More precisely, we have the following. 
	
	\begin{rem}[Invariance under reparametrization]\label{rem:invariancerepara}
		Let $0<T<\infty$ and $\gamma\colon[0,T]\times\sS\to\RR^2$ be a smooth solution to \eqref{eq:ivp}. Then
		\begin{align}
			\partial_t\gamma=-\nabla_{H_\gamma^{-1}}\sE_\lambda(\gamma)+h\partial_s\gamma
		\end{align}
		for some smooth $h \colon [0,T) \times \sS \to \RR$. 
		Let $\Phi\colon[0,T)\times\sS\to\sS$ 
		be a smooth one-parameter family of diffeomorphisms. Then the reparametrization $\tilde\gamma(t,x):=\gamma(t,\Phi(t,x))$ solves
		\begin{align}
			\partial_t\tilde\gamma
			&=\partial_t\gamma(t,\Phi(t,x))+\partial_t\Phi(t,x)\partial_x\gamma(t,\Phi(t,x))\\
			&=-\nabla_{H_\gamma^{-1}}\sE_\lambda(\gamma(t,\Phi(t,x)))+h(t,\Phi(t,x))\partial_s\gamma(t,\Phi(t,x))+\partial_t\Phi(t,x)\partial_x\gamma(t,\Phi(t,x))\\
			&=-\nabla_{H_\gamma^{-1}}\sE_\lambda(\tilde\gamma)+\tilde h\partial_s\tilde\gamma
		\end{align}
		for some smooth $\tilde h\colon[0,T)\times\sS\to\RR$. Thus, $\partial_t^\perp\tilde\gamma=-\nabla_{H_\gamma^{-1}}\sE_\lambda(\tilde\gamma)$ and $\tilde\gamma(0,\cdot)=\gamma_0(\Phi(0,\cdot))$.
	\end{rem}
	
	In what follows, we will frequently use the following lemma which gives the time derivative of various geometric quantities related to a curve $\gamma$ moving under a general flow.
	
	\begin{lem}
		\label{lem:geoprelim}
		Let $\gamma\colon[0,T)\times\sS\to\RR^2$, $T>0$, be a sufficiently smooth solution of the flow $\partial_t\gamma=V+h\partial_s\gamma$, where $V$ is the normal velocity and $h=\langle\partial_t\gamma,\partial_s\gamma\rangle$. Then
		\begin{align}
			\partial_t(\!\intd s)
			&=(\partial_sh-\langle\kappa,V\rangle)\intd s,
			\label{eq:dtds}\\
			\partial_t\partial_s-\partial_s\partial_t
			&=(\langle\kappa,V\rangle-\partial_sh)\partial_s,
			\label{eq:kommutatordtds}\\
			\partial_t\partial_s\gamma
			&=\nabla_sV+h\kappa,
			\label{eq:dttau}\\
			\partial_t\nu
			&=(-\langle\nabla_sV,\nu\rangle-h\langle\kappa,\nu\rangle)\partial_s\gamma,
			\label{eq:dtnu}\\      
			\nabla_t\kappa
			&=\nabla_s^2V+\langle\kappa,V\rangle\kappa+h\nabla_s\kappa,
			\label{eq:nablatkappa}\\
			\nabla_t\nabla_s\phi-\nabla_s\nabla_t\phi
			&=(\langle\kappa,V\rangle-\partial_sh)\nabla_s\phi+\langle\kappa,\phi\rangle\nabla_sV-\langle\nabla_sV,\phi\rangle\kappa,
			\label{eq:kommutatornablatnablas}            
		\end{align}
		for a sufficiently smooth normal field $\phi\colon[0,T)\times\sS\to\RR^2$.    
	\end{lem}
	
	Here, $\nabla_t=\partial_t^\perp=P^\perp\partial_t$. For a proof of \eqref{eq:dtds}, \eqref{eq:kommutatordtds}, \eqref{eq:dttau}, \eqref{eq:nablatkappa} and \eqref{eq:kommutatornablatnablas}, see for example \cite[Lemma 2.1]{DKS2002}. Formula \eqref{eq:dtnu} follows from \eqref{eq:dttau}.\medskip
	
	Clearly, the flow of \Cref{lem:geoprelim} does not change the rotation index $\omega$ given by
	\begin{align}
		2\pi\omega=\int_\sS k\intd s=\int_\sS\langle\kappa,\nu\rangle\intd s.
	\end{align}
	
	Next, we note the basic but crucial fact that the flow \eqref{eq:fleq} decreases the energy \eqref{eq:pee}. This follows by a direct computation using \eqref{eq:dtds}, \eqref{eq:nablatkappa} and integration by parts. All the terms containing a tangential velocity cancel.
	
	\begin{lem} \label{lem:decE}
		For a sufficently smooth solution $\gamma\colon[0,T)\times\sS\to\RR^2$, $T>0$, to \eqref{eq:fleq}, we have
		\begin{align}\label{eq:dece}
			\frac{\intd }{\intd t}\sE_\lambda(\gamma)=-\int_{\sS}\vert \nabla_s\left(\nabla_s^2\kappa+\tfrac12\vert\kappa\vert^2\kappa-\lambda\kappa\right)\vert^2\intd s\leq0.
		\end{align}
	\end{lem}
	
	\begin{rem}
		Clearly, we also observe the energy decrease in the abstract setting of \Cref{sec:deriv}, where we have
		\begin{align}
			\frac{\intd }{\intd t}\sE_\lambda(\gamma)
			&=\langle\nabla_{H_\gamma^{-1}}\sE_\lambda(\gamma),\partial_t\gamma\rangle_{H_\gamma^{-1}}
			=-\big\Vert\partial_t^\perp\gamma\big\Vert_{H_\gamma^{-1}}^2
			=-\big\Vert\nabla_s^2\nabla_{L^2}\sE_\lambda(\gamma)\big\Vert_{H_\gamma^{-1}}^2\\
			&=-\big\Vert\nabla_s\nabla_{L^2}\sE_\lambda(\gamma)\big\Vert_{L^2(\intd s)}^2,
		\end{align}
		which is \eqref{eq:dece}.
	\end{rem}
	
	At the same time, the flow \eqref{eq:fleq} preserves the signed enclosed area \eqref{eq:A}.
	
	\begin{lem} \label{lem:presA}
		Let $\gamma\colon[0,T)\times\sS\to\RR^2$ be a sufficently smooth solution to \eqref{eq:fleq}. Then
		\begin{align}
			\frac{\intd}{\intd t}\sA(\gamma)=0.
		\end{align}
	\end{lem}
	
	\begin{proof}
		Using \eqref{eq:dtnu} and integration by parts
		yields
		\begin{align}
			\frac{\intd}{\intd t}\sA(\gamma)=-\int_\sS\langle\partial_t\gamma,\nu\rangle\intd s.
		\end{align}
		So, as it should be, only the normal part of the velocity plays a role.
		With integration by parts and since $\partial_s\nu=-k\,\partial_s\gamma$ is tangential, 
		we immediately obtain
		\begin{align}
			\label{eq:presa}
			\frac{\intd}{\intd t}\sA(\gamma)
			=\int_\sS\langle\nabla_s\left(\nabla_s^2\kappa+\tfrac12\vert\kappa\vert^2\kappa-\lambda\kappa\right),\partial_s\nu\rangle\intd s=0. &\qedhere
		\end{align}
	\end{proof}
	
	A more advanced property to study is the preservation of embeddedness. For a sixth-order equation, embeddedness is generally not preserved, see \cite{Blatt}. However, by an energy argument developed in \cite{MMR2021}, the energy decay \eqref{eq:dece} suffices to provide an explicit bound on the initial energy which ensures preservation of embeddedness. Note that according to Hopf's Umlaufsatz, this only makes sense for $\omega=\pm1$.
	
	\begin{prop}\label{prop:presemb}
		Let $\omega=1$, $\lambda>0$ and $C^\ast\approx 146.628$ be as in \cite[Theorem 1.1]{MMR2021}. Let $\gamma\colon[0,T)\times\sS\to\RR^2$ be a solution to \eqref{eq:ivp} with embedded initial datum $\gamma_0$ such that 
		\begin{align}\label{eq:threshemb}
			(\sE_\lambda(\gamma_0))^2\leq 4\lambda C^\ast.
		\end{align}
		Then $\gamma$ is embedded for all $t\in[0,T)$.
	\end{prop}
	
	\begin{proof}
		If $\gamma_0$ is stationary, the claim is trivial. So with \eqref{eq:dece} we can assume that $\sE_\lambda(\gamma(t))<\sE_\lambda(\gamma_0)$ for all $t\in(0,T)$. It follows that 
		\begin{align}
			\sE(\gamma(t))\sL(\gamma(t))\leq\frac{1}{4\lambda}(\sE(\gamma(t))+\lambda\sL(\gamma(t)))^2<\frac{1}{4\lambda}(\sE_\lambda(\gamma_0))^2\leq C^\ast
		\end{align}
		for all $t\in(0,T)$. With \cite[Theorem 1.4]{MMR2021} we conclude that $\gamma$ describes an embedded curve for all $t\in(0,T)$.
	\end{proof}
	
	A short computation shows that if $\lambda\geq\frac12(\frac{\pi}{C^\ast})^2$, the above threshold is nontrivial. In this case, there exists $A_0>0$ such that a curve describing a circle with enclosed area $A_0$ satisfies \eqref{eq:threshemb}.  
	
	\section{Short time existence}
	\label{sec:shorttime}
	
	In this section we show how to prove short time existence. Addressing this question is additionally motivated by the fact that we will use several of the arguments again in the proof of convergence in \Cref{sec:conv}.  Our method of proof works not only for a smooth initial datum, but also for $\gamma_0\in C^{7,\alpha}(\sS;\RR^2)$, $\alpha>0$. This could be improved further, but since we are mainly interested in questions about the long time behavior, we do not focus on this. 
	
	\begin{thm}
		\label{thm:shorttime}
		Let $\gamma_0\colon\sS\to\RR^2$ be a smooth regular curve. Then there exists some $T>0$ and a smooth solution $\gamma\colon[0,T]\times\sS\to\RR^2$ to \eqref{eq:ivp}. This solution is unique up to reparametrizations.
	\end{thm}
	
	\begin{proof}
		The idea is to use the reparametrization argument in \Cref{sec:apprepara} and to
		find the solution as a graph over the initial datum. This approach leads to a scalar problem, which we can uniquely solve using \Cref{prop:nonlinprob}.\smallskip
		
		\textit{Step 1: Good neighborhood.}
		By \Cref{lem:repara}, there exists a $W^{5,2}$-neighborhood $V$ of the smooth immersion $\gamma_0$, such that for all $\zeta\in C^\infty(\sS;\RR^2)\cap V$, there exists a diffeomorphism $\Phi\colon\sS\to\sS$ and a reparametrization $\tilde\zeta=\zeta\circ\Phi$  
		such that $\tilde\zeta=\gamma_0+\varphi\nu_0$ for some 
		$\varphi\colon\sS\to\RR$. 
		By \Cref{rem:smoothdiffeo}, the smoothness of $\zeta$ and the smoothness of $\gamma_0$ yield $\Phi\in C^\infty(\sS;\sS) $ and $\varphi\in C^\infty(\sS;\RR)$.\smallskip
		
		\textit{Step 2: Velocity in terms of $\gamma$.}
		Let $\bar\gamma\colon\sS\to\RR^2$ be a sufficiently smooth immersed curve. Continuing the explicit calculation from \cite[Proposition A.1]{RS2020}, we obtain 
		\begin{align}
			\nabla_s^2\nabla_{L^2}\sE_\lambda(\bar\gamma)&=P^{\perp_{\bar\gamma}}\!\left(\frac{\partial_x^6\bar\gamma}{\vert\partial_x\bar\gamma\vert^6}+\tilde Q(\partial_x\bar\gamma,\dots,\partial_x^5\bar\gamma,\vert\partial_x\bar\gamma\vert^{-1}) \right)
		\end{align}
		for $\tilde Q$ a vector valued polynomial.
		If $\bar\gamma$ is time-dependent and $\bar\gamma(t,x)=\gamma_0(x)+\bar\varphi(t,x)\nu_0(x)$, then 
		$P^{\perp_{\bar\gamma}}\partial_x^6\bar\gamma$ consists of the term $\partial_x^6\bar\varphi P^{\perp_{\bar\gamma}}\nu_0$ and of terms containing derivatives of $\varphi$ up to order 5.\pagebreak
		
		As long as $\nu_0^\perp=P^{\perp_{\bar\gamma}}\nu_0$ does not vanish, we obtain
		\begin{align}\label{eq:defF}
			\nabla_s^2\nabla_{L^2}\sE_\lambda(\bar\gamma)&=\frac{\partial_x^6\bar\varphi\,\nu_0^\perp}{\vert\partial_x(\gamma_0+\bar\varphi\nu_0)\vert^6}+Q(\cdot,\bar\varphi,\partial_x\bar\varphi,\dots,\partial_x^5\bar\varphi,\vert\partial_x(\gamma_0+\bar\varphi\,\nu_0)\vert^{-1})\nu_0^\perp
		\end{align}
		with a scalar valued polynomial $Q$. Here, 
		$\gamma_0$ and its derivatives are considered as fixed functions and are included in $Q$.\smallskip
		
		\textit{Step 3: Construction of a solution.} The previous steps motivate to consider the system
		\begin{align}
			\label{eq:pdesystem2}
			\begin{cases}
				\displaystyle\partial_t\varphi= \frac{\partial_x^6\varphi}{\vert\partial_x(\gamma_0+\varphi\nu_0)\vert^6}+Q(\cdot,\varphi,\dots,\partial_x^5\varphi,\vert\partial_x(\gamma_0+\varphi\nu_0)\vert^{-1}) \quad&\text{in }(0,T)\times\sS,\\
				\displaystyle\varphi(0,\cdot)=0&\text{on }\sS.
			\end{cases}
		\end{align}
		By \Cref{prop:nonlinprob}, there exists $T>0$ and a unique solution $\varphi\in C^\infty([0,T]\times\sS;\RR^2)$ to \eqref{eq:pdesystem2}. With this solution, we define $\gamma:=\gamma_0+\varphi\nu_0\in C^\infty([0,T]\times\sS;\RR^2)$. By possibly making $T>0$ smaller, we achieve $\gamma(t,\cdot)\in V$ and $P^{\perp_\gamma}\nu_0=\nu_0^\perp\neq0$ for all $t\in[0,T]$. Moreover, 
		\begin{align}
			\partial_t^\perp\gamma=\partial_t\varphi\,\nu_0^\perp=\nabla_s^2\nabla_{L^2}\sE_\lambda(\gamma)
		\end{align}
		by \eqref{eq:defF} and \eqref{eq:pdesystem2}. Thus, $\gamma$ solves \eqref{eq:ivp}.
		\smallskip
		
		\textit{Step 4: Geometric uniqueness.} 
		Let $\tilde\gamma\in C^\infty([0,\tilde T]\times\sS;\RR^2) $ be another solution to \eqref{eq:ivp} in the sense that 
		\begin{align}
			\partial_t^\perp\tilde\gamma=\nabla_s^2\nabla_{L^2}\sE_\lambda(\tilde\gamma)
			\quad \text{and} \quad \tilde\gamma(0,\cdot)=\gamma_0\circ\tilde\Phi
		\end{align}
		for a smooth orientation preserving diffeomorphism $\tilde\Phi\colon\sS\to\sS$. For the sake of readability, we do not change the notation, but we reparametrize $\tilde\gamma$ by $\tilde\Phi^{-1}$ such that $\tilde\gamma(0,\cdot)=\gamma_0$. By \Cref{rem:invariancerepara}, the reparametrization also satisfies $\partial_t^\perp\tilde\gamma=\nabla_s^2\nabla_{L^2}\sE_\lambda(\tilde\gamma)$. Now, make $\tilde T$ so small that $\tilde\gamma(t,\cdot)\in V$ for all $t\in[0,\tilde T]$. Then there exists a reparametrization (which we again do not rename) such that $\tilde\gamma=\gamma_0+\tilde\varphi\nu_0$ for some $\tilde\varphi\in C^\infty([0,\tilde T)\times\sS;\RR)$ with $\tilde\varphi(0)=0$. Since  $\partial_t^\perp\tilde\gamma=\nabla_s^2\nabla_{L^2}\sE_\lambda(\tilde\gamma)$ and $\tilde\varphi(0)=0$, we see that $\tilde\varphi$ solves \eqref{eq:pdesystem2}. As the solution to \eqref{eq:pdesystem2} is unique, it follows that $\tilde\varphi=\varphi$ and therefore $\tilde\gamma=\gamma$. We conclude that every solution to \eqref{eq:ivp} is a reparametrization of the solution $\gamma$ constructed in Step 3.
	\end{proof}
	
	\begin{rem}\label{rem:reference}
		More generally, the solution can also be found as a graph over any arbitrary reference immersion $\hat\gamma\in W^{5,2}(\sS;\RR^2)$, provided that it satisfies $\Vert\hat\gamma-\gamma_0\Vert_{W^{4,2}}<r$, where $r$ is the constant of \Cref{lem:repara}. The reference curve in \eqref{eq:defF} and the initial datum in \eqref{eq:pdesystem2} then change accordingly.
	\end{rem}
	
	\pagebreak
	
	\section{Global existence and subconvergence}
	\label{sec:globalex}
	
	\subsection{Geometrical lemmata and interpolation inequalities}
	\label{subsec:prelim}
	
	Using \eqref{eq:dtds} and integration by parts directly yields the following lemma.
	\begin{lem}\label{lem:integraleq}
		For a sufficiently smooth solution $\gamma\colon[0,T)\times\sS\to\RR^2$ of a 
		flow $\partial_t\gamma=V+h\partial_s\gamma$ and a normal vector field $\phi:[0,T)\times\sS\to\RR^2$, we have
		\begin{align}
			\frac{\intd }{\intd t}\frac12\int_\sS\vert\phi\vert^2\intd s+\int_\sS\vert\nabla_s^3\phi\vert^2\intd s
			=\int_\sS\langle\phi,\nabla_t\phi-\nabla_s^6\phi+\frac12\phi\,\partial_sh \rangle\intd s-\frac12\int_\sS\vert\phi\vert^2\langle\kappa,V\rangle\intd s.
		\end{align}
	\end{lem}
	
	For the following, we introduce a quite common notation (see for example \cite{DKS2002} or \cite{DP14}). Let $\phi$ and $\phi_i$, $i\in\lbrace1,\dots,k\rbrace$, be normal vector fields. We denote by $\phi_1\ast\phi_2\ast\cdots\ast\phi_k$  
	a term of the type
	\begin{align}
		\begin{cases}
			\langle\phi_{i_1},\phi_{i_2}\rangle\cdots\langle\phi_{i_{k-1}},\phi_{i_k}\rangle,\quad &\text{ if }k\text{ is even,}\\
			\langle\phi_{i_1},\phi_{i_2}\rangle\cdots\langle\phi_{i_{k-2}},\phi_{i_{k-1}}\rangle\phi_{i_k},\quad &\text{ if }k\text{ is odd,}
		\end{cases}
	\end{align}
	where $i_1,\dots,i_k$ is some permutation of $1,\dots,k$. By $P^a_b(\phi)$, we denote any linear combination of terms of the type 
	\begin{align}
		\nabla_s^{j_1}\phi\ast\cdots\ast\nabla_s^{j_b}\phi\qquad\text{ with } j_k\in\NN_0 \text{ for all }k\in\lbrace1,\dots,b\rbrace\text{ and }j_1+\cdots+j_b=a
	\end{align}
	with coefficients bounded by some universal constant. Observe that
	\begin{align}\label{eq:nablasP}
		\nabla_sP_b^a(\phi)=P^{a+1}_b(\phi).
	\end{align}

	\begin{lem}
		\label{lem:phim}
		For a sufficiently smooth solution $\gamma\colon[0,T)\times\sS\to\RR^2$ of 
		$$\partial_t\gamma=\nabla_s^2\big(\nabla_s^2\kappa+\frac12\vert\kappa\vert^2\kappa-\lambda\kappa\big)+h\partial_s\gamma
		=\nabla_s^4\kappa+P^2_3(\kappa)-\lambda\nabla_s^2\kappa+h\partial_s\gamma,$$
		we have
		\begin{align}
			\nabla_t\nabla_s^m\kappa-\nabla_s^{m+6}\kappa=P^{m+4}_3(\kappa)+P^{m+2}_5(\kappa)+P_3^{m+2}(\kappa)-\lambda\nabla_s^{m+4}\kappa
			+h\nabla_s^{m+1}\kappa.
		\end{align}
	\end{lem}
	\begin{proof}
		The claim follows by induction on $m$. For $m=0$ we use \eqref{eq:nablatkappa}. The induction step follows with \eqref{eq:kommutatornablatnablas} and \eqref{eq:nablasP}.
	\end{proof}
	
	For the convenience of the reader, we repeat in the following lemmata two results of \cite{DKS2002} and \cite{DP14} without giving the proof.
	
	\begin{lem}[{\cite[Proposition 2.5]{DKS2002}}]\label{lem:interpol1}
		Let $\gamma\colon\sS\to\RR^2$ be a sufficiently smooth curve. Then for $P^a_b(\kappa)$ with $b\geq2$ containing only derivatives of $\kappa$ of order at most $k-1$ with $2k+1>a+\frac12b$, we have
		\begin{align}
			\int_\sS\vert P^a_b(\kappa)\vert\intd s\leq\varepsilon\int_\sS\vert\nabla_s^k\kappa\vert^2\intd s+C\varepsilon^{-\frac{\mu}{2-\mu}}\left(\int_\sS\vert\kappa\vert^2\intd s\right)^\frac{b-\mu}{2-\mu}+C\left(\int_\sS\vert\kappa\vert^2\intd s\right)^{a+b-1}
		\end{align}
		for any $\varepsilon>0$, $\mu=(a+\frac12b-1)/k$, and $C=C(k,a,b)$.
	\end{lem}
	
	\begin{lem}[{\cite[Lemma 4.6]{DP14}}]\label{lem:DAP12Lem4.6}
		Let $\gamma\colon\sS\to\RR^2$ be a sufficiently smooth curve. Assume that $\Vert\nabla_s^m\kappa\Vert_{L^2(\intd s)}\leq C_m$ for some $m\geq1$ and $\Vert\kappa\Vert_{L^2(\intd s)}\leq C_0$. Then
		\begin{align}
			\Vert\partial_s^l\kappa\Vert_{L^2(\intd s)}\leq C \quad\text{ for } \,0\leq l\leq m
		\end{align}
		with a constant $C=C(C_m, C_0, m, l)$.
	\end{lem}
	
	\subsection{Bounds on the length}
	\label{subsec:boundlength}
	
	To prove a longtime existence result we have to control the length of the evolving curve. A uniform bound from below is given by 
	Fenchel's theorem which yields
	\begin{align}\label{eq:lenghtlowerbound}
		4\pi^2\leq\sL(\gamma)\Vert\kappa\Vert^2_{L^2(\intd s)}\leq 2\sL(\gamma)\sE_\lambda(\gamma_0).
	\end{align}
	Here, for the second estimate, we used \Cref{lem:decE}. If $\lambda>0$, we also have a uniform bound on the length from above by 
	\begin{align}\label{eq:lengthupperbound}       \sL(\gamma)\leq\frac1\lambda\sE_\lambda(\gamma)\leq\frac1\lambda\sE_\lambda(\gamma_0).
	\end{align}
	In order to allow for $\lambda=0$, we show that the length grows at most linearly.
	
	\begin{lem}
		\label{lem:growthlength}
		Let $T>0$ and $\gamma\colon[0,T)\times\sS\to\RR^2$ be a smooth solution to \eqref{eq:ivp}. Then 
		\begin{align}
			\frac{\intd }{\intd t}\sL(\gamma)\leq C(\sE_\lambda(\gamma_0)).
		\end{align}
	\end{lem}
	
	\begin{proof}
		Using \eqref{eq:dtds} and integration by parts and applying \Cref{lem:interpol1} with $a=2$, $b=4$, $k=2$ yields 
		\begin{align}
			\frac{\intd }{\intd t}\sL(\gamma)
			&=-\int_\sS\langle\partial_t\gamma,\kappa\rangle\intd s
			=-\int_\sS\vert\nabla_s^2\kappa\vert^2\intd s+\frac32\int_\sS\vert\kappa\vert^2\vert\nabla_s\kappa\vert^2\intd s-\lambda\int_\sS\vert\nabla_s\kappa\vert^2\intd s\\
			&\leq-\int_\sS\vert\nabla_s^2\kappa\vert^2\intd s+\varepsilon\int_\sS\vert\nabla_s^2\kappa\vert^2\intd s+C(\varepsilon)\left(\int_\sS\vert\kappa\vert^2\intd s\right)^5
		\end{align}
		for all $t\in(0,T)$, where $\varepsilon>0$ can be chosen arbitrarily small. The claim follows with \Cref{lem:decE}.
	\end{proof}
	
	\subsection{Proof of global existence and subconvergence}
	\label{subsec:proofglobalex}
	
	With the preliminary work of \Cref{subsec:prelim} and \ref{subsec:boundlength} we are now ready to prove global existence for $\lambda\geq0$. If $\lambda>0$, we additionally obtain subconvergence. This result crucially relies on the fact that, in this case, the length of the evolving curve is uniformly bounded from above. 
	\begin{thm}[Global existence and subconvergence]
		\label{thm:globalexsubconv}
		For any $\lambda\geq0$, 
		there exists a geometrically unique global smooth solution $\gamma\colon[0,\infty)\times\sS\to\RR^2$ of \eqref{eq:ivp}.
		If $\lambda>0$, then as $t_i\to\infty$ the curves $\gamma(t_i,\cdot)$ subconverge smoothly, when reparametrized with constant speed and suitably translated, to some  
		stationary solution $\gamma_\infty$.
	\end{thm}
	
	\begin{proof}
		\Cref{thm:shorttime} gives a smooth solution $\gamma$ of \eqref{eq:ivp} in a small time interval. We assume by contradiction that this solution does not exist globally in time and denote by $0<T<\infty$ the maximal existence time. Let $h\in C^\infty([0,T]\times\sS;\RR)$ be the scalar tangential velocity, i.e.\ $\partial_t\gamma=\nabla_s^2\big(\nabla_s^2\kappa+\frac12\vert\kappa\vert^2\kappa-\lambda\kappa\big)+h\partial_s\gamma.$
		\smallskip
		
		\textit{Step 1: Bound on $\Vert\nabla_s^m\kappa\Vert_{L^2(\intd s)}$ for $m\in\NN_0$.}
		With \Cref{lem:integraleq} and \Cref{lem:phim} it follows that 
		\begin{align}\label{eq:global1}\begin{split}
				\frac{\intd}{\intd t}
				\int_\sS\vert\nabla_s^m\kappa\vert^2\intd s+\int_\sS\vert\nabla_s^{m+3}\kappa&\vert^2\intd s
				= \int_\sS\langle\nabla_s^m\kappa,P^{m+4}_3(\kappa)+P^{m+2}_5(\kappa)+P^{m+2}_3(\kappa)\rangle\intd s\\
				&-\lambda\int_\sS\vert\nabla_s^{m+2}\kappa\vert^2\intd s
				+\int_\sS h\langle\nabla_s^m\kappa,\nabla_s^{m+1}\kappa\rangle\intd s\\
				&+\frac12\int_\sS\vert\nabla_s^m\kappa\vert^2\partial_sh\intd s
				-\frac12\int_\sS\vert\nabla_s^m\kappa\vert^2\langle\kappa,V\rangle\intd s, 
			\end{split}
		\end{align}
		where $V=\nabla_s^2\left(\nabla_s^2\kappa+\frac12\vert\kappa\vert^2\kappa-\lambda\kappa\right)$. 
		The second integral is negative and can be ignored in the following estimates. 
		The two integrals containing the tangential velocity vanish after integration by parts.
		The other terms on the right hand side of \eqref{eq:global1} have the structure
		\begin{align}
			\begin{split}
				\int_\sS P^{2m+4}_4(\kappa)\intd s+\int_\sS P^{2m+2}_6(\kappa)\intd s+\int_\sS P^{2m+2}_4(\kappa)\intd s.
			\end{split}
		\end{align}
		With \Cref{lem:interpol1} (for $a=2m+4$, $b=4$, $k=m+3$ and for $a=2m+2$, $b=6$, $k=m+3$), we know that 
		\begin{align}\label{eq:global2}
			&\int_\sS\vert P^{2m+4}_4(\kappa)\vert\intd s, \int_\sS\vert P^{2m+2}_6(\kappa)\vert\intd s
			\leq\varepsilon\int_\sS\vert\nabla_s^{m+3}\kappa\vert^2\intd s+C(m,\varepsilon)\left(\int_{\sS}\vert\kappa\vert^2\intd s\right)^{2m+7}
		\end{align}
		for any $\varepsilon>0$. Further, 
		\Cref{lem:interpol1} with $a=2m+2$, $b=4$, $k=m+3$ yields
		\begin{align}\label{eq:global3}\begin{split}
				\int_\sS\vert P^{2m+2}_4(\kappa)\vert\intd s
				\leq&\;\varepsilon\int_\sS\vert\nabla_s^{m+3}\kappa\vert^2\intd s+C(m,\varepsilon)\left(\int_\sS\vert\kappa\vert^2\intd s\right)^{\frac23m+3}\\
				&+C(m)\left(\int_\sS\vert\kappa\vert^2\intd s\right)^{2m+5}\end{split}
		\end{align}
		for any $\varepsilon>0$. Lastly, we use \Cref{lem:interpol1} with $a=2m$, $b=2$ and $k=m+3$ to see that
		\begin{align}\label{eq:global4b}
			\int_\sS\vert\nabla_s^m\kappa\vert^2\intd s
			\leq \varepsilon\int_\sS\vert\nabla_s^{m+3}\kappa\vert^2\intd s+C\varepsilon^{-\frac{m}{3}}\int_\sS\vert\kappa\vert^2\intd s+C\left(\int_\sS\vert\kappa\vert^2\intd s\right)^{2m+1}.
		\end{align}
		Putting \eqref{eq:global1}, \eqref{eq:global2}, \eqref{eq:global3} and \eqref{eq:global4b}
		together, choosing $\varepsilon$ small enough and using \Cref{lem:decE}, we get
		\begin{align}\label{eq:globaldepTfirst}
			\frac{\intd}{\intd t}\int_\sS\vert\nabla_s^m\kappa\vert^2\intd s
			+\int_\sS\vert\nabla_s^m\kappa\vert^2\intd s
			\leq C(\sE_\lambda(\gamma_0),m).
		\end{align}  
		It follows that $\Vert \nabla_s^m\kappa\Vert_{L^2(\intd s)}(t)\leq\Vert \nabla_s^m\kappa\Vert_{L^2(\intd s)}(0) +C(\sE_\lambda(\gamma_0),m)$. 
		\smallskip
		
		\textit{Step 2: Bound on $\Vert\partial_s^m\kappa\Vert_{C^0}$ 
			for $m\in\NN_0$.} 
		By Step 1, \Cref{lem:decE} and \Cref{lem:DAP12Lem4.6} we have
		\begin{align}
			\Vert\partial_s^m\kappa\Vert_{L^2(\intd s)}\leq C(\gamma_0,m)
		\end{align}
		for $m\in\NN_0$. As in \cite[Theorem 2.2]{BGH1998} one shows that 
		\begin{align}
			\Vert\partial_s^{m}\kappa\Vert_{C^0}
			&\leq\Vert\partial_s^{m+1}\kappa\Vert_{L^1(\intd s)}+\sL(\gamma)^{-1}\Vert\partial_s^m\kappa\Vert_{L^1(\intd s)}\\
			&\leq\sL(\gamma)^\frac12\Vert\partial_s^{m+1}\kappa\Vert_{L^2(\intd s)}+\sL(\gamma)^{-\frac12}\Vert\partial_s^m\kappa\Vert_{L^2(\intd s)}.
		\end{align}
		The length is bounded in finite time by \Cref{lem:growthlength} and \eqref{eq:lenghtlowerbound}. Thus, we conclude that 
		\begin{align}\label{eq:global4}
			\Vert \partial_s^{m}\kappa\Vert_{C^0}\leq C(\gamma_0,m,T).
		\end{align}
		Note that this also implies that 
		\begin{align}\label{eq:global5}
			\Vert \partial_s^m V\Vert_{C^0}\leq C(\gamma_0,m,T).
		\end{align}
		
		\textit{Step 3: Global existence.} 
		Due to \Cref{rem:invariancerepara}, we can consider the reparametrization
		$\tilde\gamma$ of $\gamma$ with constant speed, i.e.\ such that $\vert\partial_x\tilde\gamma\vert=\sL(\gamma)$. This ensures that $\vert\partial_x\tilde\gamma\vert$ is bounded away from zero and bounded from above in the finite time interval $[0,T]$, see \eqref{eq:lenghtlowerbound} and \Cref{lem:growthlength}. Moreover, 
		\begin{align}
			\partial_x^{m+2}\tilde\gamma=\sL(\gamma)^{m+2}\partial_s^m\tilde\kappa
		\end{align}
		for $m\in\NN_0$, which implies by Step 2 that 
		$$\Vert\partial_x^{m+2}\tilde\gamma\Vert_{C^0}\leq C(\gamma_0,m,T)\quad\text{ for }m\in\NN_0.$$ 
		With the control of the normal velocity, we additionally obtain $\Vert\tilde\gamma\Vert_{C^0}\leq C(\gamma_0,T).$
		These bounds allow for a smooth extension of $\tilde\gamma$ to $[0,T]\times\sS$. The short time existence result gives a smooth extension for $t>T$ which contradicts the maximality of $T$. Thus, the maximal existence time must be infinite.\smallskip
		
		\textit{Step 4: Subconvergence for $\lambda>0$.} 
		In the case of $\lambda>0$, the length is bounded independently of $T$ (see \eqref{eq:lengthupperbound}).  
		Thus, the bounds in \eqref{eq:global4} and \eqref{eq:global5} do not depend on $T$. 
		Again, we consider the reparametrization $\tilde\gamma$ with constant speed such that
		$\vert\partial_x\tilde\gamma\vert$ is controlled up to $T=\infty$ (in this case independently of $T$). 
		Moreover, adding a suitable translation like 
		\begin{align}\label{eq:tildegamma-p}
			\tilde\gamma(t,\cdot)-p(t):=\tilde\gamma(t,\cdot)-\frac{1}{\sL(\tilde\gamma)}\int_\sS\tilde\gamma\intd s_{\tilde\gamma},
		\end{align}
		controls $\Vert\tilde\gamma-p\Vert_{C^0}$ up to $T=\infty$.
		In total, $\Vert\partial_x^m(\tilde\gamma-p)\Vert_{C^0}\leq C(\gamma_0,m)$ for all $m\in\NN_0$ and $t\in[0,\infty)$.
		So, given any sequence $(t_i)_{i\in\NN}$, $t_i\to\infty$, there exists a subsequence $(t_{i_k})_{k\in\NN}$ such that $\tilde\gamma(t_{i_k},\cdot)-p(t_{i_k})$ converges smoothly to a limit curve $\gamma_\infty$.\smallskip
		
		\textit{Step 5: Limits of convergent subsequences.} 
		We define $\tilde V:=\nabla_s^2\kappa+\frac12\vert\kappa\vert^2\kappa-\lambda\kappa$ and $u(t):=\Vert \nabla_s\tilde V\Vert^2_{L^2(\intd s)}$. With \eqref{eq:dtds}, we have
		\begin{align}
			\label{eq:step5neu1}
			\frac{\intd }{\intd t}u(t)
			=2\int_\sS\langle \nabla_t\nabla_s\tilde V, \nabla_s\tilde V\rangle\intd s
			+\int_\sS\vert\nabla_s\tilde V\vert^2\big(\partial_s\tilde h-\langle\kappa,V\rangle\big)\intd s,
		\end{align}
		where $\tilde h=\langle\partial_t\tilde\gamma,\partial_s\tilde\gamma\rangle$ is the tangential velocity of $\tilde\gamma$. With a computation using \eqref{eq:nablatkappa}, \eqref{eq:kommutatornablatnablas}, and \Cref{lem:phim}, we see that the term $\nabla_t\nabla_s\big(\vert\kappa\vert^2\kappa\big)$ appearing in $\nabla_t\nabla_s\tilde V$ can be written as
		\begin{align}
			\nabla_t\nabla_s\big(\vert\kappa\vert^2\kappa\big)=P(\kappa)+\tilde h\nabla_s^2\big(\vert\kappa\vert^2\kappa\big),
		\end{align}
		where $P(\kappa)$ represents several polynomials of the form $P^a_b(\kappa)$, $b\in\NN$ odd and $5\geq a\in\NN$. Thus, we see with \Cref{lem:phim} that 
		\begin{align}
			\nabla_t\nabla_s\tilde V=P(\kappa)+\tilde h\Big(
			\nabla_s^4\kappa+\frac12\nabla_s^2\big(\vert\kappa\vert^2\kappa\big)-\lambda\nabla_s^2\kappa\Big)=P(\kappa)+\tilde h\nabla_s^2\tilde V.
		\end{align}
		Here, $P(\kappa)$ represents now polynomials of the form $P^a_b(\kappa)$, $b\in\NN$ odd and $7\geq a\in\NN$. With this and integration by parts, we see that all terms containing the tangential velocity $\tilde h$ in \eqref{eq:step5neu1} cancel. Thus, with the previous steps, $\vert\frac{\intd }{\intd t}u(t)\vert$ is bounded by a constant depending on $\gamma_0$ for $t\in[0,\infty)$.
		Since $u\in L^1((0,\infty))$ due to \Cref{lem:decE}, this is sufficient to deduce $\lim_{t\to\infty}u(t)=0$.  
		It follows that $\gamma_\infty $ 
		solves \eqref{eq:statsolvec2} and is therefore stationary.
	\end{proof}
	
	\begin{rem}[Flow without tangential velocity]
		Like in some other articles, for example \cite{DKS2002} or \cite{DP14}, one could also impose that the tangential velocity is zero and study 
		\begin{align}\label{eq:ivp2}
			\begin{cases}
				\partial_t\gamma=V=\nabla_s^2\big(\nabla_s^2\kappa+\frac12\vert\kappa\vert^2\kappa-\lambda\kappa\big)\quad&\text{ in }(0,T)\times\sS,\\
				\gamma(0,\cdot)=\gamma_0&\text{ on }\sS.
			\end{cases}
		\end{align}
		In this case, Step 3 in the proof of global existence is different, because one does not want to work with a reparametrization. This difficulty can be overcome with the following considerations. First, one proves bounds on $\Vert\partial_x^m\kappa\Vert_{C^0}$
		for $m\in\NN_0$:
		By induction, it can be shown that for a sufficiently smooth function $h\colon\sS\to\RR$ or vector field $h\colon\sS\to\RR^2$, for example $h=\kappa$ or $h=V$, we have
		\begin{align} \label{eq:dxmh}
			\partial_x^mh=\vert\partial_x\gamma\vert^m\partial_s^mh+\sum_{j=1}^{m-1}F_{m-1}(\vert\partial_x\gamma\vert,\dots,\partial_x^{m-1}\vert\partial_x\gamma\vert)\partial_s^jh,
		\end{align}
		where $F_{m-1}$ is a polynomial of degree at most $m-1$.
		With Step 2, it therefore suffices to find a bound on $\Vert\partial_x^m\vert\partial_x\gamma\vert\Vert_{C^0}$ for $m\in\NN_0$. For this, we proceed by induction as in \cite[p.\ 649]{DP14}. For $m=0$, we note that $\partial_t\vert\partial_x\gamma\vert=-\langle\kappa,V\rangle\vert\partial_x\gamma\vert$ by \eqref{eq:dtds}. Since $\Vert\langle\kappa,V\rangle\Vert_{C^0}$ is bounded by \eqref{eq:global4} and \eqref{eq:global5} and $\vert\partial_x\gamma_0\vert$ is bounded from below and above, we conclude that 
		\begin{align}\label{eq:regular}
			C^{-1}<\vert\partial_x\gamma\vert<C \quad\text{ for some } C=C(\gamma_0,T).
		\end{align}
		For the induction step from $m-1$ to $m$ one uses \eqref{eq:dxmh} for $h=\langle\kappa,V\rangle$ to obtain bounds $\Vert\partial_x^i\langle\kappa,V\rangle\Vert_{C^0}\leq C(\gamma_0,m,T)$ for $0\leq i\leq m$. 
		Differentiating $\partial_t\vert\partial_x\gamma\vert=-\langle\kappa,V\rangle\vert\partial_x\gamma\vert$, we thus obtain
		$$\partial_t\partial_x^m\vert\partial_x\gamma\vert\leq-\langle\kappa,V\rangle\partial_x^m\vert\partial_x\gamma\vert+C(\gamma_0,m,T)$$
		which yields the bound on $\Vert\partial_x^m\vert\partial_x\gamma\vert\Vert_{C^0}$, that we wanted to find. Note that the bounds on $\Vert\partial_x^m\kappa\Vert_{C^0}$ also yield bounds on $\Vert\partial_x^mV\Vert_{C^0}$ for $m\in\NN_0$. Again, the uniform bounds on $\kappa$, $V$, $\gamma$ and all their derivatives yield a smooth extension up to the maximal existence time and beyond. 
	\end{rem}
	
	\section{Convergence}
	\label{sec:conv}
	
	\subsection{Stationary solutions as constrained critical points}
	\label{subsec:constcritp}
	
	Working with curves, that enclose a given area $A_0\in\RR$, we introduce the constraint functional
	\begin{align}
		\sG\colon W^{2,2}_{\mathrm{Imm}}(\sS;\RR^2)\to\RR,\qquad
		\sG(\gamma)=\sA(\gamma)-A_0=-\frac12\int_\sS\langle\nu_\gamma,\gamma\rangle\intd s_\gamma-A_0.
	\end{align}
	Here, $W^{2,2}_{\mathrm{Imm}}(\sS;\RR^2)$ denotes the space of all $W^{2,2}$-immersions. 
	A short computation shows that the Fréchet derivative is given by
	\begin{align}
		\sG'(\gamma)\colon W^{2,2}(\sS;\RR^2)\to\RR,\qquad
		\sG'(\gamma)(Y)=-\int_\sS\langle\nu_\gamma,Y\rangle\intd s_\gamma,
	\end{align}
	which is a surjective mapping. Hence, the energy $\sE_\lambda\colon W^{2,2}_{\mathrm{Imm}}(\sS;\RR^2)\to\RR$ has a critical point subject to the constraint $\sG=0$ at $\gamma\in W^{2,2}_{\mathrm{Imm}}(\sS;\RR^2)\cap\lbrace\sG=0\rbrace$ if and only if there exits $c_1\in\RR$ such that 
	\begin{align}
		\label{eq:weakEL}
		0&=\sE'(\gamma)(Y)+c_1\sG'(\gamma)(Y)\\
		&=\int_\sS\Big(\langle\nabla_{s_\gamma}^2Y,\kappa_\gamma\rangle
		-\frac12\vert\kappa_\gamma\vert^2\langle\partial_{s_\gamma}Y,\partial_{s_\gamma}\gamma\rangle+\lambda\langle\partial_{s_\gamma}Y,\partial_{s_\gamma}\gamma\rangle\Big)\intd s_\gamma
		-c_1\int_\sS\langle\nu_\gamma,Y\rangle\intd s_\gamma
	\end{align}
	for all $Y\in W^{2,2}(\sS;\RR^2)$. For 
	a computation of $\sE'(\gamma)$, see for example \cite[Lemma A.1]{DP14}.
	If $\gamma$ is sufficiently regular, we obtain from \eqref{eq:weakEL} the Euler Lagrange equation 
	\begin{align}
		\label{eq:strongEL}
		0=\nabla_s^2\kappa+\frac12\vert\kappa\vert^2\kappa-\lambda\kappa-c_1\nu
		=\nabla_{L^2}\sE_\lambda(\gamma)-c_1\nabla_{L^2}\sG(\gamma),
	\end{align}
	where $\kappa=\kappa_\gamma$ and $\nu=\nu_\gamma$. 
	Since a stationary solution is a smooth regular curve satisfying \eqref{eq:statsolvec2}, we see that for any stationary solution, there is $c_1\in\RR$ such that \eqref{eq:strongEL} holds. Thus, any stationary solution is an (area-)constrained critical point.
	
	\subsection{A constrained {\L}ojasiewicz--Simon gradient inequality}
	\label{subsec:Loja}
	
	The proof of the convergence result in \Cref{subsec:proofconv} will mainly relay on a constrained version of the \L ojasiewicz--Simon gradient inequality. 
	We will use this inequality for the $L^2$-gradients of the energy functional $\sE_\lambda$ and the constraint functional $\sG$. So, we work in the domain of these gradients. 
	Moreover, we first work with variations normal to a given regular curve. With this in mind, we define the following Hilbert spaces and functionals. 
	
	\begin{defn}
		\label{def:VandH}
		Let $\hat\gamma\in W^{4,2}_{\mathrm{Imm}}(\sS;\RR^2)$. The space of normal vector fields along $\hat\gamma$ is denoted by
		\begin{align}
			V^\perp:=W^{4,2,\perp}(\sS;\RR^2)=\lbrace Y\in W^{4,2}(\sS;\RR^2):\langle Y,\partial_x\hat\gamma\rangle =0\rbrace
		\end{align}
		(cf.\ \eqref{eq:Vperp}).
		Moreover, we set
		\begin{align}
			L^{2,\perp}(\sS;\RR^2)=\lbrace Y\in L^2(\sS;\RR^2): \langle Y,\partial_x\hat\gamma\rangle =0 \text{ a.e.}\rbrace.
		\end{align}
		Let $\varepsilon>0$ be so small such that $\hat\gamma+Y$ is immersed for all $Y\in U_\varepsilon:=B_\varepsilon(0)\subset V^\perp$. With this, we define
		\begin{align}
			E_\lambda&\colon U_\varepsilon\to\RR, \qquad E_\lambda(Y):=\sE_\lambda(\hat\gamma+Y),\label{eq:defE}\\
			G&\colon U_\varepsilon\to\RR, \qquad G(Y):=\sG(\hat\gamma+Y).\label{eq:defG}
		\end{align}
	\end{defn}
	
	As in \cite{DPS16} or in \cite{RS2020}, we first prove a \L ojasiewicz--Simon gradient inequality in $V^\perp$, i.e.\ for normal directions. Since we work in a constrained setting, the inequality will be based on \cite{ConstrLoja}. The following two lemmata cover the assumptions in  \cite[Cor. 5.2]{ConstrLoja}.
	
	\begin{lem}
		\label{lem:ass5.2E}
		Let $E_\lambda$ be defined as in \eqref{eq:defE}. 
		Then the following holds.
		\begin{enumerate}[1.]
			\item[1.] The energy $E_\lambda\colon U_\varepsilon\to\RR$ is analytic. 
			\item[2.] The $L^{2,\perp}$-gradient $d E_\lambda\colon U_\varepsilon\to L^{2,\perp}(\sS;\RR^2)$,
			\begin{align}
				Y\mapsto P^\perp\Big(\nabla_{s_{\hat\gamma+Y}}^2\kappa_{\hat\gamma+Y}+\frac12\vert\kappa_{\hat\gamma+Y}\vert^2\kappa_{\hat\gamma+Y}-\lambda\kappa_{\hat\gamma+Y}\Big)\vert\partial_x(\hat\gamma+Y)\vert,
			\end{align}
			with $P^\perp=\langle\cdot,\nu_{\hat\gamma}\rangle\nu_{\hat\gamma}$, is analytic.
			\item[3.] The derivative of the gradient evaluated at zero $(d E_\lambda)'(0)\colon V^\perp\to L^{2,\perp}$ is Fredholm of index zero.
		\end{enumerate}
	\end{lem}
	
	For the definition and fundamental properties of analytic maps between Banach spaces, see for example 
	\cite{W1965}. A detailed proof of \Cref{lem:ass5.2E} in the case of open elastic curves can be found in \cite[Section 3]{DPS16}. By slightly changing the function spaces (no boundary conditions but working in $\sS$), this proof is transferred to the case of closed curves. A proof of 3.\ in \Cref{lem:ass5.2E} is also given in \cite{MantegazzaPozzetta21}. We remark that for this Fredholm property of the second variation, it is crucial that we consider only variations in normal direction. \smallskip
	
	For the constrained version of the \L ojasiewicz--Simon gradient inequality, we additionally have to analyze the functional $G$.
	
	\begin{lem}
		\label{lem:ass5.2G}
		The constraint functional $G$ defined in \eqref{eq:defG} satisfies the following.
		\begin{enumerate}[1.]
			\item The functional $G\colon U_\varepsilon\to\RR$ is analytic. 
			\item The $L^{2,\perp}$-gradient $dG\colon U_\varepsilon\to L^{2,\perp}(\sS;\RR^2)$, 
			\begin{align}
				Y\mapsto-P^\perp\nu_{\hat\gamma+Y}\vert\partial_x(\hat\gamma+Y)\vert=
				-\langle\nu_{\hat\gamma+Y},\nu_{\hat\gamma}\rangle\nu_{\hat\gamma}\,\vert\partial_x(\hat\gamma+Y)\vert
			\end{align}
			is analytic. 
			\item The derivative of the gradient evaluated at zero $(d G)'(0)\colon V^\perp\to L^{2,\perp}$ is compact.
			\item If $\hat\gamma$ in \Cref{def:VandH} is given such that $\sA(\hat\gamma)=A_0$, we have $G(0)=0$ and $dG(0)\neq0$.
		\end{enumerate}
	\end{lem}
	
	\begin{proof}
		\textit{1. and 2.} As in \cite[Lem.\ 3.4]{DPS16} and \cite[Prop.\ 4.5]{RS2020}, the functions
		\begin{align}
			U_\varepsilon\to C(\sS;\RR), \;Y\mapsto \vert\partial_x(\hat\gamma+Y)\vert\quad\text{ and }\quad
			\ U_\varepsilon\to W^{3,2}(\sS;\RR^2), \; Y\mapsto\partial_{s_{\hat\gamma+Y}}(\hat\gamma+Y)
		\end{align}
		are analytic. Since rotation by a fixed angle is a welldefined, continuous, linear operator on $W^{3,2}(\sS;\RR^2)\subset C^2(\sS;\RR^2)$, 
		\begin{align}
			U_\varepsilon\to W^{3,2}(\sS;\RR^2), \quad Y\mapsto\nu_{\hat\gamma+Y}
		\end{align}
		is also analytic. The function $Y\mapsto\hat\gamma+Y$ is analytic as a map from $U_\varepsilon$ to $W^{3,2}(\sS;\RR^2)$. Thus, the scalar product
		\begin{align}
			U_\varepsilon\to W^{3,2}(\sS;\RR), \quad Y\mapsto\langle\hat\gamma+Y,\nu_{\hat\gamma+Y}\rangle
		\end{align} 
		is analytic.
		As $W^{3,2}(\sS;\RR^2)$ is continuously embedded into $L^2(\sS;\RR^2)$, $Y\mapsto\langle\hat\gamma+Y,\nu_{\hat\gamma+Y}\rangle$ is also analytic as a map from $U_\varepsilon$ to $L^2(\sS;\RR^2)$.
		Since scalar multiplication $L^2(\sS;\RR)\times C(\sS;\RR)\to L^2(\sS;\RR)$ is analytic, so is the map 
		\begin{align}
			U_\varepsilon\to L^2(\sS;\RR), \quad Y\mapsto\langle\hat\gamma+Y,\nu_{\hat\gamma+Y}\rangle\vert\partial_x(\hat\gamma+Y)\vert.
		\end{align} 
		Since integration over $\sS$ is a welldefined continuous and linear operator on $L^2(\sS;\RR)$, the analyticity of $G\colon U_\varepsilon\to \RR$ follows. Similar arguments and the continuity and linearity of the projection $P^\perp\colon L^2(\sS;\RR^2)\to L^{2,\perp}(\sS;\RR^2)$ yields the analyticity of the $L^{2,\perp}$-gradient $dG\colon U_\varepsilon\to L^{2,\perp}(\sS;\RR^2)$.\\
		\textit{3.} With the projection $P^\perp=\langle\cdot,\nu_{\hat\gamma}\rangle\nu_{\hat\gamma}$, we compute
		\begin{align}
			(dG')(0)Y
			&=\frac{\intd }{\intd t}\Big\vert_{t=0}\;dG(tY)
			=-P^\perp\frac{\intd }{\intd t}\Big\vert_{t=0}\;\nu_{\hat\gamma+tY}\;\vert\partial_x\hat\gamma\vert-P^\perp\nu_{\hat\gamma}\,\frac{\intd }{\intd t}\Big\vert_{t=0}\;\vert\partial_x(\hat\gamma+tY)\vert\\
			&=\nu_{\hat\gamma}\vert\partial_x\hat\gamma\vert\langle Y,\kappa_{\hat\gamma}\rangle.\label{eq:dG'0}
		\end{align}
		Here, we used that 
		\begin{align}
			\frac{\intd }{\intd t}\Big\vert_{t=0}\;\partial_{s_{\hat\gamma+tY}}(\hat\gamma+tY)=\nabla_{s_{\hat\gamma}}Y
		\end{align}
		(see for example \cite[(B.5)]{DPS16}), which implies that the first summand in \eqref{eq:dG'0} vanishes. For the second summand, we used
		\begin{align}
			\frac{\intd }{\intd t}\Big\vert_{t=0}\;\vert\partial_x(\hat\gamma+tY)\vert=-\vert\partial_x\hat\gamma\vert\langle Y,\kappa_{\hat\gamma}\rangle
		\end{align}
		(see for example \cite[(B.2)]{DPS16}). We see that $(dG')(0)\colon V^\perp\to L^{2,\perp}$ is of zeroth order in $Y$. 
		Hence it is compact. \\
		\textit{4.} The assumptions on $\hat\gamma$ imply that $G(0)=\sG(\hat\gamma)=0$. Moreover, $\vert\partial_x\hat\gamma\vert\neq0$ and $\nu_{\hat\gamma}\neq0$. Thus, $dG(0)=\nu_{\hat\gamma}\vert\partial_x\hat\gamma\vert\neq0$.
	\end{proof}
	
	With this, we prove a first version of a constrained \L ojasiewicz-Simon gradient inequality which reads as follows.
	
	\begin{thm}\label{thm:Lojanormal}
		Let $\hat\gamma\in C^\infty_{\mathrm{Imm}}(\sS;\RR^2)$ be a constrained critical point of $\sE_\lambda$. Then there exist constants $C>0$, $\vartheta\in(0,\frac12]$ and $r>0$, such that 
		\begin{align}
			\vert E_\lambda(0)-E_\lambda(Y)\vert^{1-\vartheta}
			\leq C\Vert P(Y)dE_\lambda(Y)\Vert_{L^2(\intd x)}
		\end{align}
		for all $Y\in B_r(0)\subset U_\varepsilon\subset V^\perp$ with $G(Y)=0$. Here, $P(Y)\colon L^{2,\perp}(\sS;\RR^2)\to L^{2,\perp}(\sS;\RR^2)$ is the orthogonal projection onto $\lbrace X\in L^{2,\perp}(\sS;\RR^2): \langle\nu_{\hat\gamma+Y},X\rangle_{L^2(\intd s_{\hat\gamma+Y})}=0\rbrace$.
	\end{thm}
	
	\begin{proof}
		We check that the assumptions in \cite[Corollary 5.2]{ConstrLoja} are satisfied for $E_\lambda$ and $G$ defined on $U_\varepsilon\subset V^\perp$. In \Cref{lem:ass5.2E} and \Cref{lem:ass5.2G} we have seen that $E$ and $G$ are analytic. Assumption (i), i.e.\ $V^\perp(\sS;\RR^2)\hookrightarrow L^{2,\perp}(\sS;\RR^2)$ densely, is given. Assumptions (ii) and (iii) follow from \Cref{lem:ass5.2E} and assumptions (iv) to (vi) from \Cref{lem:ass5.2G}. Thus, by \cite[Corollary 5.2]{ConstrLoja}, the set of area preserving normal variations
		\begin{align}
			\mathcal{M}:=\lbrace Y\in U_\varepsilon:G(Y)=0\rbrace
		\end{align}
		is locally a submanifold of codimension $1$ near $0$. Since $\hat\gamma$ is a constrained critical point of $\sE_\lambda$, $0$ is a critical point of $E_\lambda$ on $\mathcal{M}$. Thus, \cite[Corollary 5.2]{ConstrLoja} further yields the existence of constants $C,r>0$ and $\vartheta\in(0,\tfrac12]$ such that for all $Y\in\mathcal{M}$ with $\Vert Y\Vert_{W^{4,2}(\intd x)}\leq r$, it holds that
		\begin{align}
			\vert E(Y)-E(0)\vert^{1-\vartheta}\leq C\Vert P(Y)dE(Y)\Vert_{L^2(\intd x)},
		\end{align}
		where $P(Y)\colon L^{2,\perp}(\sS;\RR^2)\to L^{2,\perp}(\sS;\RR^2)$ is the orthogonal projection onto the closure of the tangent space $\overline{T_Y\mathcal{M}}^{\Vert\cdot\Vert_{L^2(\intd x)}}$. With \cite[Proposition 3.3]{ConstrLoja} it follows that 
		\begin{align}
			\overline{T_Y\mathcal{M}}^{\Vert\cdot\Vert_{L^2(\intd x)}}
			&=\lbrace X\in L^{2,\perp}(\sS;\RR^2): \langle dG(Y),X\rangle_{L^2(dx)}=0\rbrace\\
			&=\lbrace X\in L^{2,\perp}(\sS;\RR^2): \langle\nu_{\hat\gamma+Y},X\rangle_{L^2(\intd s_{\hat\gamma+Y})}=0\rbrace.\qedhere
		\end{align}
	\end{proof}

	To prove a version of the \L ojasiewicz--Simon inequality for variations that are not necessarily normal, the reparametrization argument in \Cref{sec:apprepara} is crucial. 
	
	With \Cref{lem:repara}, \Cref{thm:Lojanormal} can be extended to area preserving variations with tangential component.
	
	\begin{thm}\label{thm:Loja}
		Let the immersion $\hat\gamma\in C^\infty(\sS;\RR^2)$ be a constrained critical point of $\sE_\lambda$.
		Then there exist constants $C>0$, $\vartheta\in(0,\tfrac12]$ and $\tilde r>0$ such that 
		\begin{align}
			\vert\sE_\lambda(\gamma)-\sE_\lambda(\hat\gamma)\vert^{1-\vartheta}
			\leq C\,\Big\Vert\nabla_{L^2}\sE_\lambda(\gamma)-\frac{\nabla_{L^2}\sG(\gamma)}{\sL(\gamma)}\int_\sS\langle\nabla_{L^2}\sE_\lambda(\gamma),\nu_\gamma\rangle\intd s_\gamma\Big\Vert_{L^2(\intd s_\gamma)}
		\end{align}
		for all $\gamma\in W^{4,2}(\sS;\RR^2)$ with $\sG(\gamma)=0$ and $\Vert\gamma-\hat\gamma\Vert_{W^{4,2}}\leq \tilde r$. 
	\end{thm}
	
	\begin{proof}
		Let $C>0$, $\vartheta\in(0,\frac12]$ and $r>0$ be given as in \Cref{thm:Lojanormal}. Moreover, let $\tilde r=\tilde r(\hat\gamma,r)$ be the constant of \Cref{lem:repara}. We take some $\gamma\in W^{4,2}(\sS;\RR^2)$ with $\sG(\gamma)=0$ and $\Vert\gamma-\hat\gamma\Vert_{W^{4,2}}\leq \tilde r$. By \Cref{lem:repara}, there exists a $W^{4,2}$-diffeomorphism $\Phi\colon\sS\to\sS$ and $Y\in B_r(0)\subset V^\perp$ such that $\gamma\circ\Phi=\hat\gamma+Y$. Since $0=\sG(\gamma)
		=\sG(\gamma\circ\Phi)$, we know that $G(Y)=0$. Thus, the invariance of $\sE_\lambda$ under reparametrization and \Cref{thm:Lojanormal} yields
		\begin{align}\label{eq:prfloja1}
			\vert\sE_\lambda(\gamma)-\sE_\lambda(\hat\gamma)\vert^{1-\vartheta}
			&= \vert\sE_\lambda(\hat\gamma+Y)-\sE_\lambda(\hat\gamma)\vert^{1-\vartheta}
			= \vert E_\lambda(Y)-E_\lambda(0)\vert^{1-\vartheta}\\
			&\leq C\Vert P(Y)dE_\lambda(Y)\Vert_{L^2(\intd x)}.
		\end{align}
		By \cite[Remark 5.4]{ConstrLoja}, the projection operator $P(Y)\colon L^{2,\perp}(\sS;\RR^2)\to L^{2,\perp}(\sS;\RR^2)$ is given by
		\begin{align}
			P(Y)X=X-\frac{\langle dG(Y),X\rangle_{L^2(\intd x)}}{\Vert dG(Y)\Vert^2_{L^2(\intd x)}}d G(Y).
		\end{align}
		In particular, we have $P(Y)(c_1 dG(Y))=0$ for
		\begin{align}
			c_1=c_1(\gamma)&:=-\frac{1}{\sL(\gamma)}\int_\sS\langle\nabla_{L^2}\sE_\lambda(\gamma),\nu_\gamma\rangle\intd s_\gamma. 
		\end{align}
		Moreover, comparing $dE_\lambda$ and $dG$ to the $L^2$-gradients, we have
		\begin{align} dE_\lambda(Y)=P^\perp\left(\nabla_{L^2}\sE_\lambda(\hat\gamma+Y)\right)\vert\partial_x(\hat\gamma+Y)\vert,\quad
			dG(Y)=P^\perp\left(\nabla_{L^2}\sG(\hat\gamma+Y)\right)\vert\partial_x(\hat\gamma+Y)\vert,
		\end{align}
		where $P^\perp=\langle\cdot,\nu_{\hat\gamma}\rangle\nu_{\hat\gamma}$. With these two observations, we obtain
		\begin{align}
			\Vert P(Y)dE_\lambda(Y)\Vert_{L^2(\intd x)}
			&=\Vert P(Y)(dE_\lambda(Y)+c_1dG(Y))\Vert_{L^2(\intd x)}\\
			&\leq \Vert dE_\lambda(Y)+c_1dG(Y)\Vert_{L^2(\intd x)}\label{eq:prfloja2}\\
			&=\Vert P^\perp\left(\nabla_{L^2}\sE_\lambda(\hat\gamma+Y)+c_1\nabla_{L^2}\sG(\hat\gamma+Y)\right)\vert\partial_x(\hat\gamma+Y)\vert \Vert_{L^2(\intd x)}.
		\end{align}
		By Sobolev embedding and possibly making $r$ smaller, we can assume that $\vert\partial_x(\hat\gamma+Y)\vert$ is uniformly bounded for $\Vert Y\Vert_{W^{4,2}}\leq r$. So, putting \eqref{eq:prfloja1} and \eqref{eq:prfloja2} together yields 
		\begin{align}
			\vert\sE_\lambda(\gamma)-\sE_\lambda(\hat\gamma)\vert^{1-\vartheta}
			&\leq C 
			\Vert \nabla_{L^2}\sE_\lambda(\gamma\circ\Phi)+c_1\nabla_{L^2}\sG(\gamma\circ\Phi) \Vert_{L^2(\intd s_{\gamma\circ\Phi})}
		\end{align}
		for a modified constant $C>0$. It remains to notice that the gradients are geometric, i.e.
		\begin{align}
			\nabla_{L^2}\sE_\lambda(\gamma\circ\Phi)=\nabla_{L^2}\sE_\lambda(\gamma)\circ\Phi, \quad
			\nabla_{L^2}\sG(\gamma\circ\Phi)=\nabla_{L^2}\sG(\gamma)\circ\Phi.
		\end{align}
		With this, we obtain
		\begin{align}
			\Vert \nabla_{L^2}\sE_\lambda(\gamma\circ\Phi)+c_1\nabla_{L^2}\sG(\gamma\circ\Phi) \Vert_{L^2(\intd s_{\gamma\circ\Phi})}
			=\Vert \nabla_{L^2}\sE_\lambda(\gamma)+c_1\nabla_{L^2}\sG(\gamma) \Vert_{L^2(\intd s_{\gamma})}
		\end{align} 
		and the claim follows. 
	\end{proof}

	\subsection{Proof of full convergence}
	\label{subsec:proofconv}
	We are now prepared to enhance the subconvergence result to a full convergence result. This finally proves \Cref{thm:combined}.
	
	\begin{thm}[Convergence]
		\label{thm:convergence}
		Let $\lambda>0$ and $\gamma\colon[0,\infty)\times\sS\to\RR^2$ be the global smooth solution of \eqref{eq:ivp}. Then there exists a family of smooth diffeomorphisms $\Phi(t)\colon\sS\to\sS$, $t\in(0,\infty)$, such that $\gamma(t,\Phi(t,\cdot))$ converges smoothly for $t\to\infty$ to a stationary solution $\gamma_\infty$.
	\end{thm}
	
	For the proof of full convergence, we use an argumentation similar to \cite[Section 5]{DPS16}, but due to the $H^{-1}$-structure of the flow, applying the \L ojasiewicz inequality is more challenging and time-dependent norms have to be handled. In \cite{GG21}, the authors have to face similar difficulties. We note that in \cite{RS2020}, an elegant way to shorten the proof of convergence in \cite[Section 5]{DPS16} is shown, but it seems that this cannot be transferred to $H^{-1}$-flows.
	
	\begin{proof}[Proof of \Cref{thm:convergence}] We take a smooth solution $\gamma\colon[0,\infty)\times\sS\to\RR^2$ of \Cref{thm:globalexsubconv} and $\tilde\gamma(t,\cdot)-p(t)$ of \eqref{eq:tildegamma-p}, which is the reparametrized with constant speed and translated solution.
		\smallskip
		
		\textit{Step 1: Convergent sequence.} 
		By the subconvergence result (\Cref{thm:globalexsubconv}), there exists a sequence of times $t_i\to\infty$, such that $\tilde\gamma(t_i,\cdot)-p(t_i)$ converges smoothly to a stationary solution $\gamma_\infty$ (which will be kept fix throughout the proof). As discussed in \Cref{subsec:constcritp}, $\gamma_\infty$ is an area constrained critical point of $\sE_\lambda$. Due to \eqref{eq:dece}, we know that 
		\begin{align}\label{eq:deceinconv}
			\sE_\lambda(\gamma(t))=\sE_\lambda(\tilde\gamma(t)-p(t))\geq\sE_\lambda(\gamma_\infty) \quad\text{for all } t\in[0,\infty).
		\end{align}
		We assume that the inequality is strict. If not, there exists $\bar t\in[0,\infty)$ such that $\sE_\lambda(\gamma(\bar t))=\sE_\lambda(\gamma_\infty)$. But due to \Cref{lem:decE}, this means $\sE_\lambda(\gamma(t))=\sE_\lambda(\gamma_\infty)$ for all $t\geq\bar t$, so $\partial_t^\perp\gamma\equiv0$ on $[\bar t,\infty)$ by \eqref{eq:dece}. 
		Thus, only the tangential velocity changes on on $[\bar t,\infty)$
		and the claimed convergence follows trivially.
		\smallskip
		
		\textit{Step 2: Fixing $\delta=\delta(\gamma_\infty)$.}
		For the following, we fix $\delta>0$ such that for all $\zeta\colon\sS\to\RR^2$ with $\Vert\zeta-\gamma_\infty\Vert_{C^1}<\delta$, it holds that
		\begin{align}\label{eq:M}
			\vert\partial_x\zeta\vert\geq M:=\frac12\min\vert\partial_x\gamma_\infty\vert=\frac{\sL(\gamma_\infty)}{2}>0,
		\end{align}
		see \eqref{eq:lenghtlowerbound}, and such that 
		\begin{align}\label{eq:alphaperp}
			\vert\alpha^\perp\vert:=\Big\vert\alpha-\Big\langle\alpha,\frac{\partial_x\zeta}{\vert\partial_x\zeta\vert}\Big\rangle\frac{\partial_x\zeta}{\vert\partial_x\zeta\vert}\Big\vert\geq\frac12\vert\alpha\vert
		\end{align}
		for all vector fields $\alpha$ normal to $\gamma_\infty$. The latter inequality can be achieved by
		\begin{align}
			\Big\vert\Big\langle\alpha,\frac{\partial_x\zeta}{\vert\partial_x\zeta\vert}\Big\rangle\Big\vert=\Big\vert\Big\langle\alpha, \frac{\partial_x\zeta}{\vert\partial_x\zeta\vert}-\frac{\partial_x\gamma_\infty}{\vert\partial_x\gamma_\infty\vert}\Big\rangle\Big\vert
		\end{align}
		becoming small for small distance 
		$\Vert\zeta-\gamma_\infty\Vert_{C^1}$.
		\smallskip
		
		\textit{Step 3: New initial datum and new flow.}
		Let $0<\varepsilon<\frac12\min\lbrace\delta,\tilde r,1\rbrace$, where $\tilde r$ is given as in \Cref{thm:Loja} for $\hat\gamma=\gamma_\infty$. Then \Cref{lem:repara} yields the existence of $\tilde\varepsilon=\tilde\varepsilon(\varepsilon,\gamma_\infty)>0$ such that for all $X\in B_{\tilde\varepsilon}(0)\subset W^{7,2}(\sS;\RR^2)$, 
		\begin{align}
			(\gamma_\infty+X)\circ\Phi=\gamma_\infty+Y
		\end{align}
		for a $W^{7,2}$-diffeomorphism $\Phi$ and $Y\in B_\varepsilon(0)\subset W^{7,2}(\sS;\RR^2)$ normal to $\gamma_\infty$. 
		Due to the smooth subconvergence, there is an index $i\in\NN$ such that $\Vert\tilde\gamma(t_i)-p(t_i)-\gamma_\infty\Vert_{W^{7,2}}<\tilde\varepsilon$.
		Thus, there exists a $W^{7,2}$-diffeomorphism $\Phi_0$ such that 
		\begin{align}\label{eq:phi0}
			(\tilde\gamma(t_i)-p(t_i))\circ\Phi_0=\gamma_\infty+N_0
		\end{align}
		for some $N_0\in W^{7,2}(\sS;\RR^2)$ normal to $\gamma_\infty$ with
		\begin{align}\label{eq:normN0}
			\Vert N_0\Vert_{W^{7,2}}<\varepsilon<\frac12\min\lbrace\delta,\tilde r\rbrace.
		\end{align}
		For this index $i$, we set $\bar\gamma_0:=(\tilde\gamma(t_i)-p(t_i))\circ\Phi_0=\gamma_\infty+N_0$. 
		Possibly choosing $\varepsilon>0$ smaller, we can ensure that $\Phi_0$ does not change the orientation.
		So $\bar\gamma_0$ describes a smooth closed curve with $\sA(\bar\gamma_0)=A_0$. 
		Moreover, by \Cref{rem:smoothdiffeo}, $\Phi_0$ and $N_0$ are smooth. 
		
		The idea now is to take $\bar\gamma_0$ as new initial datum and to consider 
		\begin{align}\label{eq:newsystem}
			\begin{cases}
				\partial^\perp_t\bar\gamma=\nabla_s^2\nabla_{L^2}\sE_\lambda(\bar\gamma)\qquad&\text{ in }(0,T)\times\sS,\\
				\bar\gamma(0,\cdot)=\bar\gamma_0=\gamma_\infty+N_0=\gamma_\infty+\varphi_0\nu_\infty &\text{ on }\sS.
			\end{cases}
		\end{align}
		By \Cref{thm:shorttime} and \Cref{rem:reference}, there exists a a geometrically unique smooth solution $\bar\gamma$ of \eqref{eq:newsystem} that can be written as 
		$\bar\gamma(t,x)=\gamma_\infty(x)+\varphi(t,x)\nu_\infty(x)$ for some smooth scalar function $\varphi$.
		Let $0<T'\leq\infty$ the maximal existence time of this solution.
		The geometric uniqueness implies that on $[t_i,T')$, $\bar\gamma$ is a smooth reparametrization of $\tilde\gamma-p(t_i)$.
		It follows that for proving \Cref{thm:convergence}, it suffices to prove that $T'=\infty$ and that $\bar\gamma$ converges to $\gamma_\infty$ as $t\to\infty$. 
		\smallskip
		
		\textit{Step 4:}
		Let $T\in(0,T']$ be maximal such that 
		\begin{align}\label{eq:normvarphi}
			\Vert N(t)\Vert_{W^{6,2}}:=\Vert\bar\gamma(t)-\gamma_\infty\Vert_{W^{6,2}}<\min\lbrace\delta,\tilde r\rbrace\quad\text{ for all } t\in[0,T). 
		\end{align}
		Due to \eqref{eq:normN0}, such a $T$ exists.
		Our goal is to show $T=T'=\infty$. 
		By parabolic Schauder estimates (as in \cite[Section 4]{CFS09}), it can be shown that
		\begin{align}\label{eq:adrian}
			\sup_{[0,T)}\Vert\bar\gamma(t)-\gamma_\infty\Vert_{C^{6,\alpha}}\leq C(\delta,\tilde r,\gamma_0)
		\end{align}
		for some $\alpha\in(0,\frac12)$.
		\smallskip
		
		\textit{Step 5: Application of the \L ojasiewicz-Simon inequality.}
		We define 
		$$\mathcal{H}(t):=(\sE_\lambda(\bar\gamma(t))-\sE_\lambda(\gamma_\infty))^\vartheta \quad\text{ for } t\in(0,T).$$ 
		Here $\vartheta$ is given as in \Cref{thm:Loja} for $\gamma_\infty$. Note that since $\bar\gamma$ is a reparametrization and translation of $\gamma$ and we made the assumption that the inequality in \eqref{eq:deceinconv} is strict, $\mathcal{H}(t)\neq0$ for all $t$. With \eqref{eq:newsystem} and integration by parts, we compute
		\begin{align}
			-\frac{\intd }{\intd t}\mathcal{H}(t)
			&=-\vartheta\big(\sE_\lambda(\bar\gamma(t))-\sE_\lambda(\gamma_\infty)\big)^{\vartheta-1}\,\frac{\intd }{\intd t}\sE_\lambda(\bar\gamma(t))\\
			&=-\vartheta\big(\sE_\lambda(\bar\gamma(t))-\sE_\lambda(\gamma_\infty)\big)^{\vartheta-1}\,\langle\nabla_{L^2}\sE_\lambda(\bar\gamma(t)),\partial_t\bar\gamma(t)\rangle_{L^2(\intd s_{\bar\gamma})}\\
			&=-\vartheta\big(\sE_\lambda(\bar\gamma(t))-\sE_\lambda(\gamma_\infty)\big)^{\vartheta-1}\,\langle\nabla_{L^2}\sE_\lambda(\bar\gamma(t)),\partial_t^\perp\bar\gamma(t)\rangle_{L^2(\intd s_{\bar\gamma})}\\
			&=\vartheta\big(\sE_\lambda(\bar\gamma(t))-\sE_\lambda(\gamma_\infty)\big)^{\vartheta-1}\,\Vert\nabla_s\nabla_{L^2}\sE_\lambda(\bar\gamma(t))\Vert^2_{L^2(\intd s_{\bar\gamma})}.
		\end{align}
		Since $\bar\gamma$ solves \eqref{eq:newsystem} with initial datum $\bar\gamma_0$ with $\sG(\bar\gamma_0)=0$, we know that $\sG(\bar\gamma)=0$. Thus, on $(0,T)$, we can apply \Cref{thm:Loja} and obtain
		\begin{align}
			-\frac{\intd }{\intd t}\mathcal{H}(t)
			\geq C\vartheta\,\frac{\Vert\nabla_s\nabla_{L^2}\sE_\lambda(\bar\gamma(t))\Vert^2_{L^2(\intd s_{\bar\gamma})}}{\Vert\nabla_{L^2}\sE_\lambda(\bar\gamma(t))-c_1(\bar\gamma(t))\bar\nu\Vert_{L^2(\intd s_{\bar\gamma})}},\label{eq:fastfeierabend1}
		\end{align}
		where
		\begin{align}
			c_1(\bar\gamma(t))=\frac{1}{\sL(\bar\gamma(t))}\int_\sS\langle\nabla_{L^2}\sE_\lambda(\bar\gamma(t)),\bar\nu\rangle\intd s_{\bar\gamma}.
		\end{align}
		Here, remember that $\nabla_{L^2}\sG(\bar\gamma)=-\bar\nu$ and note that $\Vert\nabla_{L^2}\sE_\lambda(\bar\gamma(t))-c_1(\bar\gamma(t))\bar\nu\Vert_{L^2(\intd s_{\bar\gamma})}>0$ by \Cref{thm:Loja} and $\sE_\lambda(\bar\gamma(t))>\sE_\lambda(\gamma_\infty)$.
		Writing the denominator of \eqref{eq:fastfeierabend1} in terms of the scalar curvature $\bar k=\langle\bar\kappa,\bar\nu\rangle$ and using the Poincaré--Wirtinger inequality, we obtain
		\begin{align}
			&\Vert\nabla_{L^2}\sE_\lambda(\bar\gamma(t))-c_1(\bar\gamma(t))\bar\nu\Vert_{L^2(\intd s_{\bar\gamma})}\\
			&\quad=\Big\Vert\partial_s^2\bar k+\frac12\bar k^3-\lambda\bar k-\frac{1}{\sL(\bar\gamma)}\int_\sS\Big(\partial_s^2\bar k+\frac12\bar k^3-\lambda\bar k\Big)\intd s_{\bar\gamma}\Big\Vert_{L^2(\intd s_{\bar\gamma})}\\
			&\quad\leq C(\sL({\bar\gamma}))\Big\Vert\partial_s\Big(\partial_s^2\bar k+\frac12\bar k^3-\lambda\bar k\Big)\Big\Vert_{L^2(\intd s_{\bar\gamma})}
			=C\Vert\nabla_s\nabla_{L^2}\sE_\lambda(\bar\gamma(t))\Vert_{L^2(\intd s_{\bar\gamma})}.\label{eq:lastcor}
		\end{align}
		Here, we used that $\langle\nabla_s^n\kappa,\nu\rangle=\partial_s^nk$ for all $n\in\NN_0$ and $\langle\nabla_s^2(\vert\kappa\vert^2\kappa),\nu\rangle=\partial_s^2k^3$.
		With \eqref{eq:lastcor}, we get
		\begin{align}\label{eq:w1}
			-\frac{\intd }{\intd t}\mathcal{H}(t)\geq C\vartheta\Vert\nabla_s\nabla_{L^2}\sE_\lambda(\bar\gamma(t))\Vert_{L^2(\intd s_{\bar\gamma})}
			=C\vartheta\Vert\partial_t^\perp\bar\gamma\Vert_{H_{\bar\gamma}^{-1}}
		\end{align}
		for a constant $C$ depending on bounds on $\sL(\gamma)$ and the constant of \Cref{thm:Loja}. The definition of the latter norm is given in \Cref{subsec:Hgamma}, see \eqref{eq:normHgamma-1}.
		\smallskip
		
		\textit{Step 6: Time-independent norm.}
		With the same partial integration as in \Cref{lem:presA}, we see that $\nabla_s^2\nabla_{L^2}\sE_\lambda(\bar\gamma)\in H_{\bar\gamma}$ and hence $\partial_t^\perp\bar\gamma\in H_{\bar\gamma}$. So \Cref{lem:auxnorm} implies that 
		\begin{align}\label{eq:w2}
			\Vert\partial_t^\perp\bar\gamma\Vert_{H_{\bar\gamma}^{-1}}\geq C\Vert\partial_t^\perp\bar\gamma\Vert_\star
		\end{align}
		for a constant $C$ depending on $\Vert\bar\kappa\Vert_{C^0}$, $\sL(\bar\gamma)$ as well as $\Vert\vert\partial_x\bar\gamma\vert^{-1}\Vert_{C^0}$ and $\Vert\partial_x\bar\gamma\Vert_{C^0}$. We already know that $\bar\gamma$ is a reparametrization and translation of $\gamma$ and from the proof of \Cref{thm:globalexsubconv}, we know that $\Vert\kappa\Vert_{C^0}$ and $\sL(\gamma)$ are uniformly bounded. Moreover, $\sup\vert\partial_x\bar\gamma\vert^{-1}$ is bounded by \eqref{eq:M} and \eqref{eq:normvarphi} and $\sup\vert\partial_x\bar\gamma\vert\leq\sup\vert\partial_x\bar\gamma-\partial_x\gamma_\infty\vert+\vert\partial_x\gamma_\infty\vert$ is bounded by \eqref{eq:normvarphi}. Thus, the constant $C$ does not depend on the choice of $\varepsilon$. 
		Moreover, since $\partial_t\bar\gamma$ is orthogonal to $\gamma_\infty$, we have
		\begin{align}
			&\Vert\partial_t\bar\gamma\Vert_\star
			=\sup_{Y\in W^{1,2}\setminus\lbrace0\rbrace}\;
			\frac{1}{\Vert Y\Vert_{W^{1,2}}}\int_\sS\langle\partial_t\bar\gamma,Y\rangle\intd x
			=\sup_{Y\in W^{1,2}\setminus\lbrace0\rbrace}\;
			\frac{1}{\Vert Y\Vert_{W^{1,2}}}\int_\sS\langle\partial_t\bar\gamma,Y^{\perp_\infty}\rangle\intd x\\
			&\leq\sup_{Y\in W^{1,2}\setminus\lbrace0\rbrace}\;
			\frac{1}{\Vert Y\Vert_{W^{1,2}}}\int_\sS\langle\partial_t^\perp\bar\gamma,Y^{\perp_\infty}\rangle\intd x
			+\sup_{Y\in W^{1,2}\setminus\lbrace0\rbrace}\;
			\frac{1}{\Vert Y\Vert_{W^{1,2}}}\int_\sS\langle\partial_t\bar\gamma,\partial_s\bar\gamma\rangle\langle Y^{\perp_\infty},\partial_s\bar\gamma\rangle\intd x\\
			&\leq \sup_{Y\in W^{1,2,\perp_\infty}\setminus\lbrace0\rbrace}\;
			\frac{C(\gamma_\infty)}{\Vert Y\Vert_{W^{1,2}}}\int_\sS\langle\partial_t^\perp\bar\gamma,Y\rangle\intd x
			+\sup_{Y\in W^{1,2}\setminus\lbrace0\rbrace}\;
			\frac{1}{\Vert Y\Vert_{W^{1,2}}}\int_\sS
			\langle\partial_t\bar\gamma,Y\rangle(\langle \nu_\infty,\partial_s\bar\gamma\rangle)^2\intd x
		\end{align}
		where $Y^{\perp_\infty}=\langle Y,\nu_\infty\rangle\nu_\infty$ and $W^{1,2,{\perp_\infty}}=W^{1,2}\cap\lbrace Y=Y^{\perp_\infty}\rbrace$. 
		For the first term in the last inequality, we used a computation as in \eqref{eq:komanna} and for the second term we used $\partial_t\bar\gamma=\langle\partial_t\bar\gamma,\nu_\infty\rangle\nu_\infty$.
		Since $1=\vert\nu_\infty\vert^2=\vert\langle\nu_\infty,\bar\nu\rangle\vert^2+\vert\langle\nu_\infty,\partial_s\bar\gamma\rangle\vert^2$ and $\vert\langle\nu_\infty,\bar\nu\rangle\vert\geq\frac12$ on $(0,T)$ by \eqref{eq:alphaperp}, we conclude that
		\begin{align}\label{eq:w3}
			\Vert\partial_t\bar\gamma\Vert_\star\leq C(\gamma_\infty)\Vert\partial_t^\perp\bar\gamma\Vert_\star+\frac34\Vert\partial_t\bar\gamma\Vert_\star
		\end{align}
		on $(0,T)$.
		Putting \eqref{eq:w1}, \eqref{eq:w2} and \eqref{eq:w3} together yields
		\begin{align}\label{eq:w8}
			-\frac{\intd }{\intd t}\mathcal{H}(t)\geq C\Vert\partial_t\bar\gamma\Vert_\star
		\end{align}
		for a constant $C$ depending on the bounds of the solution $\gamma$ from the proof of \Cref{thm:globalexsubconv}, the constants from \Cref{thm:Loja}, $\delta$ and $T$. Integration now yields that on $(0,T)$, we have
		\begin{align}
			\Vert\bar\gamma(t)-\gamma_\infty\Vert_\star
			&\leq\Vert\bar\gamma(t)-\bar\gamma(0)\Vert_\star+\Vert\bar\gamma(0)-\gamma_\infty\Vert_\star
			\leq\int_0^t\Vert\partial_t\bar\gamma\Vert_\star\intd \tau+\Vert\bar\gamma(0)-\gamma_\infty\Vert_\star\\
			&\leq C\Big(\big(\sE_\lambda(\bar\gamma(0))-\sE_\lambda(\gamma_\infty)\big)^\vartheta-\big(\sE_\lambda(\bar\gamma(t))-\sE_\lambda(\gamma_\infty)\big)^\vartheta\Big)+\Vert\bar\gamma(0)-\gamma_\infty\Vert_\star\\
			&\leq C\Vert\bar\gamma(0)-\gamma_\infty\Vert_{C^2}^\vartheta+\Vert\bar\gamma(0)-\gamma_\infty\Vert_\star
			\leq C\Vert\bar\gamma(0)-\gamma_\infty\Vert_{C^2}^\vartheta.\label{eq:w4}
		\end{align}
		Here, we used $\Vert\bar\gamma(0)-\gamma_\infty\Vert_\star\leq\Vert\bar\gamma(0)-\gamma_\infty\Vert_{L^2(\gamma_\infty)}\leq C(\sL(\gamma_\infty))\Vert\bar\gamma(0)-\gamma_\infty\Vert_{C^0}$, \eqref{eq:normN0}, and $\vartheta<1$.
		\smallskip
		
		\textit{Step 7: Interpolation.}
		Since $N$ is normal to $\gamma_\infty$ and since $\sA(\bar\gamma(t))=\sA(\gamma_\infty)=A_0$ implies $\int_\sS\langle N(t),\nu_\infty\rangle\intd s_{\gamma_\infty}=0$, we have $\bar\gamma(t)-\gamma_\infty=N(t)\in H_{\gamma_\infty}$. Thus, \Cref{lem:interpol2} and \Cref{lem:auxnorm} yield
		\begin{align}
			\Vert\bar\gamma(t)-\gamma_\infty\Vert_{L^2(\intd s_{\gamma_\infty})}^2
			&\leq\Vert\bar\gamma(t)-\gamma_\infty\Vert_{H_{\gamma_\infty}}\Vert\bar\gamma(t)-\gamma_\infty\Vert_{H^{-1}_{\gamma_\infty}}\\
			&\leq C\Vert\bar\gamma(t)-\gamma_\infty\Vert_{H_{\gamma_\infty}}\Vert\bar\gamma(t)-\gamma_\infty\Vert_\star\\
			&\leq C\Vert\bar\gamma(t)-\gamma_\infty\Vert_{H_{\gamma_\infty}}\Vert\bar\gamma(0)-\gamma_\infty\Vert_{C^2}^\vartheta
			\leq C\delta\Vert\bar\gamma(0)-\gamma_\infty\Vert_{C^2}^\vartheta,\label{eq:w5}
		\end{align}
		where we used \eqref{eq:w4} and $\Vert\bar\gamma(t)-\gamma_\infty\Vert_{H_{\gamma_\infty}}\leq C \Vert\bar\gamma(t)-\gamma_\infty\Vert_{W^{6,2}}<\delta$, see \eqref{eq:normvarphi}. With \cite[Prop.\ 1.1.3]{L1995}, we have
		\begin{align}
			\Vert\bar\gamma(t)-\gamma_\infty\Vert_{W^{6,2}}\leq\Vert\bar\gamma(t)-\gamma_\infty\Vert_{C^6}
			\leq C \Vert\bar\gamma(t)-\gamma_\infty\Vert_{C^0}^{1-\frac{6}{6+\alpha}}\Vert\bar\gamma(t)-\gamma_\infty\Vert_{C^{6+\alpha}}^{\frac{6}{6+\alpha}}\label{eq:gk3}
		\end{align}
		for $\alpha\in(0,\frac12)$. Moreover, by Sobolev embedding and with interpolation (see for example \cite[Thm.\ 6.4.5 (3)]{BL2012}),
		\begin{align}
			\Vert\bar\gamma(t)-\gamma_\infty\Vert_{C^0}
			&\leq\Vert\bar\gamma(t)-\gamma_\infty\Vert_{C^{0,\frac12}}
			\leq C\Vert\bar\gamma(t)-\gamma_\infty\Vert_{W^{1,2}}\\
			&\leq C\Vert\bar\gamma(t)-\gamma_\infty\Vert^\frac12_{L^2}\Vert\bar\gamma(t)-\gamma_\infty\Vert^\frac12_{W^{2,2}}.\label{eq:gk1}
		\end{align}
		Putting the last two equations together and using \eqref{eq:adrian} yields
		\begin{align}
			\Vert\bar\gamma(t)-\gamma_\infty\Vert_{W^{6,2}}
			&\leq C\Vert\bar\gamma(t)-\gamma_\infty\Vert_{L^2}^{\frac12-\frac{6}{2(6+\alpha)}}\Vert\bar\gamma(t)-\gamma_\infty\Vert_{W^{2,2}}^{\frac12-\frac{6}{2(6+\alpha)}}\Vert\bar\gamma(t)-\gamma_\infty\Vert_{C^{6+\alpha}}^{\frac{6}{6+\alpha}}\\
			&\leq C(\tilde r,\delta,\gamma_0)\Vert\bar\gamma(t)-\gamma_\infty\Vert_{L^2(\intd s_{\gamma_\infty})}^\beta\label{eq:gk4}
		\end{align}
		for some $\beta\in(0,1)$. With \eqref{eq:w5} and \eqref{eq:normN0}, we obtain
		\begin{align}\label{eq:w6}
			\Vert\bar\gamma(t)-\gamma_\infty\Vert_{W^{6,2}}
			\leq C \Vert\bar\gamma(0)-\gamma_\infty\Vert_{C^2}^{\frac{\vartheta\beta}{2}}<C\varepsilon^{\frac{\vartheta\beta}{2}}. 
		\end{align}
		
		\textit{Step 8: $T=T'=\infty$.}
		Choosing $\varepsilon<\big(\frac{\min\lbrace\delta,\tilde r\rbrace}{2C}\big)^\frac{2}{\vartheta\beta}$, we obtain from \eqref{eq:w6} 
		\begin{align}
			\Vert\bar\gamma(t)-\gamma_\infty\Vert_{W^{6,2}}<\frac{\min\lbrace\delta,\tilde r\rbrace}{2}\label{eq:gk2}
		\end{align}
		on $(0,T)$.
		This gives a contradiction to the maximality of $T$ if $T<T'$, see \eqref{eq:normvarphi}. Thus, we obtain $T=T'$ and \eqref{eq:normvarphi} holds up to $T'$. Assume that $T'<\infty$. Since \eqref{eq:adrian} holds up to $t=T'$, we can restart \eqref{eq:pdesystem2} with initial datum determined by $\bar\gamma(T')$. This contradiction to the maximality of $T'$ yields $T'=\infty$ and 
		\begin{align}\label{eq:w7}
			\Vert\bar\gamma(t)-\gamma_\infty\Vert_{W^{6,2}}<\min\lbrace\delta,\tilde r\rbrace \qquad\text{ for all } t\in(0,\infty).
		\end{align}
		
		\textit{Step 9: Convergence.}
		With \eqref{eq:w7} and \eqref{eq:w8}, we obtain $\Vert\partial_t\bar\gamma\Vert_\star\in L^1(\RR_+)$. From
		\begin{align}
			\Vert\bar\gamma(t)-\bar\gamma(t')\Vert_{L^2}^2
			&\leq
			\Vert\bar\gamma(t)-\bar\gamma(t')\Vert_{H_{\gamma_\infty}}\Vert\bar\gamma(t)-\bar\gamma(t')\Vert_{H_{\gamma_\infty}^{-1}}
			\leq C\delta\int_{t'}^t\Vert\partial_t\bar\gamma\Vert_\star\intd \tau\to0\quad\text{ for }t',t\to\infty,
		\end{align}
		it follows that $\lim_{t\to\infty}\bar\gamma(t)$ exists in $L^2(\sS;\RR^2)$ and thus equals $\gamma_\infty$. By a subsequence argument, we obtain $\Vert\bar\gamma(t)-\gamma_\infty\Vert_{C^k}\to0$ as $t\to\infty$ for any $k\in\NN$, i.e.\ the claimed smooth convergence follows. 
	\end{proof}

	\subsection{Stationary solutions}
	\label{sec:statsol}
	
	The convergence result outlined in \Cref{subsec:proofconv} raises interest in gaining a deeper understanding of stationary solutions. 
	
	From \eqref{eq:statsolvec2},
	we see that $\gamma$ is a stationary solution if and only if there exists a constant $c_1\in\RR$ such that 
	\begin{align}\label{eq:statsolc1vec}
		\nabla_s^2\kappa+\tfrac12\vert\kappa\vert^2\kappa-\lambda\kappa=c_1\nu.
	\end{align}
	If $c_1=0$ and $\lambda>0$, we obtain the classical elasticae equation \eqref{eq:ELelasticae}, for which the solutions are classified. 
	So in a certain sense, the constant $c_1$ measures the deviation from the stationary solution to an elastica. Remember that the only closed elasticae are multifold coverings of the circle and of the figure eight elastica, see \cite[Lemma 5.4]{MR2021}. 
	
	For $c_1\neq0$, \eqref{eq:statsolc1vec} is the equation which determines the equilibrium shapes of cylindrical fluid membranes exposed to a pressure difference $c_1$ and experiencing a tensile stress $\lambda$, see \cite{Z97} or \cite{VDM08}. 
	In this context, \eqref{eq:statsolc1vec} is the Euler--Lagrange equation corresponding to the problem of minimizing the elastic energy
	\begin{align}\label{eq:area penalized energy}
		\frac12\int_\gamma\vert\kappa\vert^2\intd s+\lambda\sL(\gamma)+c_1\sA(\gamma)
	\end{align}
	with penalized area ($\lambda>0$ given) and penalized length ($c_1>0$ given). 
	In \cite{Z97}, a complete classification of the critical points corresponding to closed curves is presented. Similarly, \cite{VDM08} explicitly provides the solutions to \eqref{eq:statsolc1vec} in a format comparable to \cite{Z97}. The authors state a condition under which the solution describes a closed curve and investigate when the curve is simple. 
	By \cite[Lemma 5.4]{DP2017}, minimizing \eqref{eq:area penalized energy} is equivalent to minimize $\sE_\lambda$ under the constraint of fixed enclosed area. The latter problem is studied in \cite{W2000} and \cite{WT2008} for $\omega=1$. 
	In particular, we note that from the mentioned articles, it follows that for $c_1\neq0$, there exist solutions of \eqref{eq:statsolc1vec} that do not correspond to a multifold covered circle or a multifold covered figure eight. Concrete examples can be found in \cite{Z97}, \cite{VDM08} or \cite{WT2008}. We note that the references mentioned are by no means complete. For further references, see also the citations within the mentioned articles.
	
	If $A_0\neq0$, the defect $c_1$ in \eqref{eq:statsolc1vec} can be measured in terms of the unpenalized elastic energy and the length of the curve. 
	
	\begin{lem}\label{lem:c1}
		Let $\gamma$ be a stationary solution and $c_1\in\RR$ such that $\gamma$ satisfies \eqref{eq:statsolc1vec}. Then
		\begin{align}\label{eq:c1}
			2c_1A_0=\sE(\gamma)-\lambda\sL(\gamma).
		\end{align}
	\end{lem}
	
	\begin{proof}
		We assume that $\gamma\colon\sS\to\RR^2$ satisfies \eqref{eq:statsolc1vec} on $\sS$. This implies
		\begin{align}\label{eq:marrakech}
			\int_\sS\langle\nabla_s^2\kappa,\gamma\rangle\intd s
			+\frac12\int_\sS\vert\kappa\vert^2\langle\kappa,\gamma\rangle\intd s-\lambda\int_\sS\langle\kappa,\gamma\rangle\intd s
			=c_1\int_\sS\langle\gamma,\nu\rangle\intd s.
		\end{align}
		The first integral can be written as
		\begin{align}
			\int_\sS\langle\nabla_s^2\kappa,\gamma\rangle\intd s
			&=-\int_\sS\langle\partial_s\nabla_s\kappa,\partial_s\gamma\rangle\langle\gamma,\partial_s\gamma\rangle\intd s
			=\int_\sS\langle\nabla_s\kappa,\kappa\rangle\langle\gamma,\partial_s\gamma\rangle\intd s\\
			&=-\frac12\int_\sS\vert\kappa\vert^2\intd s-\frac12\int_\sS\vert\kappa\vert^2\langle\kappa,\gamma\rangle\intd s.
		\end{align}
		Putting this into \eqref{eq:marrakech} and with another integration by parts in the third integral of \eqref{eq:marrakech} yields \eqref{eq:c1}.   
	\end{proof}
	
	\begin{rem}[Equipartition of energy]
		If $A_0=0$ and $\gamma$ is a stationary solution, \Cref{lem:c1} implies that 
		\begin{align}
			\sE(\gamma)=\lambda\sL(\gamma),
		\end{align}
		where $\sE$ is the unpenalized elastic energy \eqref{eq:ee}. So, the two parts of the penalized elastic energy \eqref{eq:pee} are equally weighted. 
	\end{rem}
	
	The dynamic approach with the $H_{\gamma}^{-1}$-gradient flow provides an elegant proof of the existence of solutions to \eqref{eq:statsolc1vec} that are no elastica. With our approach, there is no need for an explicit parametrization using Jacobi elliptic functions (as in \cite{Z97}, \cite{VDM08} and \cite{WT2008}), but simply an appropriate choice of the initial datum.
	
	\begin{cor}[of \Cref{thm:combined}, Existence of non-elastica stationary solutions]
		There exists $c_1\in\RR$ and solutions to \eqref{eq:statsolc1vec} that do not describe a multifold covered circle or figure eight elastica.
	\end{cor}
	
	We consider a concrete example. Let $\lambda>0$, $\omega=0$ and $A_0\neq0$. Clearly, for this choice of parameters, there exists an admissible initial datum, see for example \Cref{fig:statsolomega0Aneq0}. By \Cref{thm:convergence}, the solution to \eqref{eq:fleq} with this initial datum converges to a stationary solution. Since the rotation index and the enclosed signed area are preserved along the flow, this stationary solution satisfies $\omega=0$ and $A_0\neq0$. The only closed elastica with $\omega=0$ is the figure eight elastica with $\sA=0$. Thus, the stationary solution does not describe an elastica. The same arguments apply for $\sA=0$ and arbitrary $\omega\neq0$. \Cref{fig:statsolomega1A0} gives an exemplary initial datum. 
	
	\begin{figure}[h]
		\centering
		\subfloat[][]{\label{fig:statsolomega0Aneq0}
			\includegraphics[width=0.44\textwidth]{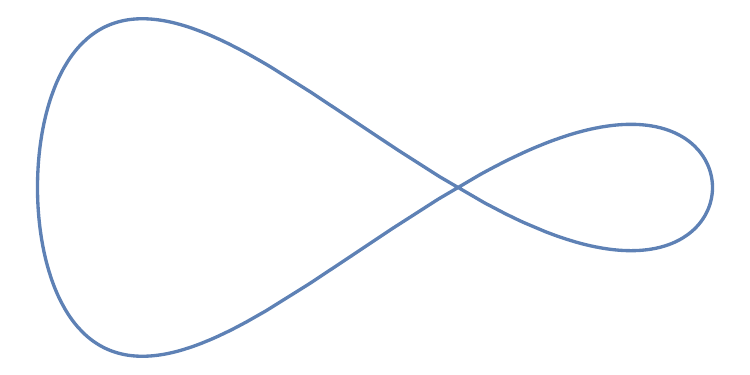}\vspace{1.5em}
		}
		\qquad\quad
		\hfill
		\subfloat[][]{\label{fig:statsolomega1A0}
			\includegraphics[width=0.44\textwidth]{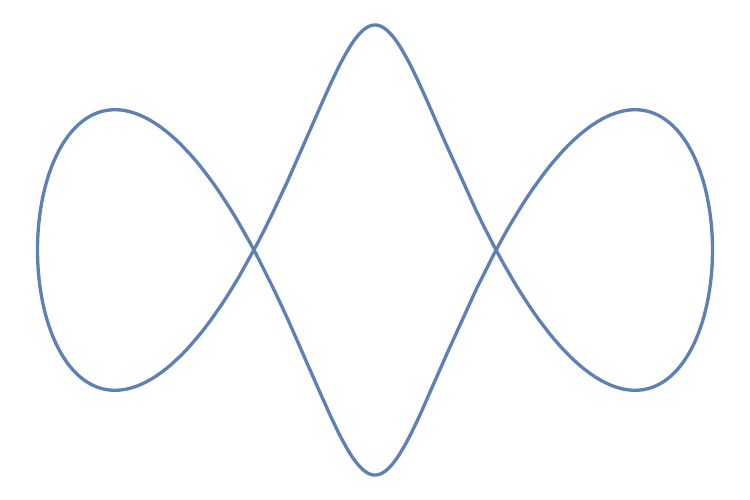}
		}
		\caption{Exemplary initial data for $\omega=0$ and $\sA\neq0$ in \subref{fig:statsolomega0Aneq0} and for $\omega=1$ and $\sA=0$ in \subref{fig:statsolomega1A0}.}
		\label{fig:nontrivstatsol}
	\end{figure}
	
	\subsection{Blow-up in infinite time}
	\label{sec:nonconvergent}
	Without penalizing the length, solutions do not converge in general. A simple example for arbitrary rotation index is given by the following situation. 
	
	\begin{lem}\label{lem:exnonconv}
		If $A_0=0$ and $\lambda=0$, the length $\sL(\gamma)$ of the global solution $\gamma$ from \Cref{thm:globalexsubconv} becomes unbounded 
		as time goes to infinity.
	\end{lem}
	
	\begin{proof}
		Assume there is $M\in\RR$ such that $\sL(\gamma(t))<M$ for all $t\geq0$. Then, proceeding as in the proof of \Cref{thm:globalexsubconv} yields subconvergence. As in \Cref{subsec:proofconv}, this improves to full convergence to a stationary solution $\gamma_\infty$. This stationary solution satisfies \eqref{eq:c1} for a constant $c_1\in\RR$, i.e.\ we have
		\begin{align}
			0=2 c_1A_0=\sE(\gamma_\infty)-\lambda\sL(\gamma_\infty)=\sE(\gamma_\infty).
		\end{align}
		This implies $\kappa_\infty\equiv0$, which contradicts the closedness of $\gamma_\infty$.
	\end{proof}
	
	We note that also for $A_0\neq0$ and $\lambda=0$, the length of the curve is not controlled by a bound on the elastic energy in general, see for example \cite{BuHe17}.
	
	\appendix
	
	\section{Details for the short time existence}
	\label{sec:appste}

	\subsection{A crucial reparametrization argument}
	\label{sec:apprepara}
	For the convenience of the reader, we state the following reparametrization argument, which can be found with a detailed proof in \cite[Lemma 4.1 and Remark 4.2]{DPS16}. In addition to the proof of short time existence (\Cref{thm:shorttime}), it is also used in the proof of the constrained \L ojasiewicz--Simon gradient inequality (\Cref{thm:Loja}).
	
	\begin{lem}\label{lem:repara}
		Let $k\geq4$ and $\hat\gamma\in W^{k+1,2}(\sS;\RR^2)$ be a regular curve. 
		We define 
		\begin{align}
			V^\perp:=W^{4,2,\perp}(\sS;\RR^2)=\lbrace Y\in W^{4,2}(\sS;\RR^2):\langle Y,\partial_x\hat\gamma\rangle =0\rbrace.\label{eq:Vperp}
		\end{align}
		Then, there exists a constant $r=r(\hat\gamma)>0$ such that 
		for all $X\in B_{r}(0)\subset W^{k,2}(\sS;\RR^2)$, there exists a $W^{k,2}$-diffeomorphism $\Phi\colon\sS\to\sS$ such that 
		\begin{align}\label{eq:repara}
			(\hat\gamma+X)\circ\Phi=\hat\gamma+Y
		\end{align}
		for some $Y\in V^\perp\cap W^{k,2}(\sS;\RR^2)$.\\
		For $r>0$ given, there exists
		$\tilde r=\tilde r(\hat\gamma,r)>0$ such that for all $X\in B_{\tilde r}(0)\subset W^{k,2}(\sS;\RR^2)$, 
		\eqref{eq:repara} holds
		for some $Y\in B_r(0)\subset V^\perp\cap W^{k,2}(\sS;\RR^2)$.
	\end{lem}
	
	\begin{rem}\label{rem:smoothdiffeo}
		With the notations in the preceding lemma, it can be shown that smoothness of the reference curve $\hat\gamma$ and smoothness of $X\in B_r(0)$ yields smoothness of the diffeomorphism $\Phi$ and the normal variation $Y$. For details on this, see \cite[Proof of Lemma 4.1]{DPS16}.
	\end{rem}
	
	\subsection{Parabolic Hölder spaces and the linear problem}
	\label{sec:apphoelder}
	
	For the proof of short time existence, we work with parabolic Hölder spaces of order $6$, cf.\ \cite[Section II.3.1]{EZ1998} or \cite{LSU1988}. 
	
	Let $l\in\RR_+\!\setminus\ZZ$, 
	$a,b\in\RR$ with $a<b$, $T>0$ and let $Q_T:=(0,T)\times(a,b)$. The parabolic Hölder space $H^{\frac{l}{6},l}\left(\overline{Q_T}\right)$ is the space of all functions $f\colon\overline{Q_T}\to\RR$ with continuous derivatives $\partial_t^m\partial_x^nf$ for $6m+n<l$ and finite norm
	\begin{align}
		\left\Vert f\right\Vert_{H^{\frac{l}{6},l}\left(\overline{Q_T}\right)}
		:=\,&\sum_{j=0}^{\lfloor l \rfloor}\sum_{6m+n=j}\sup_{(t,x)\in{Q_T}}\left\vert\partial_t^m\partial_x^n f(t,x)\right\vert\\
		&+\sum_{6m+n=\lfloor l \rfloor}\sup_{\substack{(t,x_1),(t,x_2)\in Q_T\\ x_1\neq x_2}}\frac{\left\vert \partial_t^m\partial_x^nf(t,x_1)-\partial_t^m\partial_x^nf(t,x_2)\right\vert}{\left\vert x_1-x_2\right\vert^{l-\lfloor l \rfloor}}\\
		&+\sum_{0<l-6m-n<6}\sup_{\substack{(t_1,x),(t_2,x)\in Q_T \\ t_1\neq t_2}}\frac{\left\vert \partial_t^m\partial_x^nf(t_1,x)-\partial_t^m\partial_x^nf(t_2,x)\right\vert}{\left\vert t_1-t_2\right\vert^{\frac{l-6m-n}{6}}}.
	\end{align}
	To define the parabolic Hölder space $H^{\frac{l}{6},l}\left([0,T]\times\sS\right)$, we identify $f\colon[0,T]\times\sS\to\RR$ with its periodic counterpart $f\colon[0,T]\times\RR\to\RR$. By choosing $a,b$ such that the interval $(a,b)$ contains two periods of $f$, $H^{\frac{l}{6},l}\left([0,T]\times\sS\right)$ consists of the functions with periodic counterpart in $H^{\frac{l}{6},l}\left(\overline{Q_T}\right)$.\bigskip
	
	We first consider the linear parabolic problem
	\begin{align}
		\begin{cases}\label{eq:linprob}
			L\varphi:=\partial_t\varphi-\sum_{q=0}^6a_q\partial_x^q\varphi=f\quad&\text{ in }(0,T]\times\sS,\\
			\varphi(0,\cdot)=\varphi_0&\text{ on }\sS
		\end{cases}
	\end{align}
	for some $f\in H^{\frac{l}{6},l}\left([0,T]\times\sS\right)$, coefficients $a_q\in H^{\frac{l}{6},l}\left([0,T]\times\sS\right)$ and $\varphi_0\in C^{6+l}(\sS)$.
	
	\begin{prop}\label{prop:linprob}
		Let $L$ be parabolic, $l\in\RR_+\setminus\ZZ$ and assume that $a_q\in H^{\frac{l}{6},l}\left([0,T]\times\sS\right)$, $q=0,\ldots,6$. Then for all $f\in H^{\frac{l}{6},l}\left([0,T]\times\sS\right)$ and $\varphi_0\in C^{6+l}(\sS)$, there exist a unique solution $\varphi\in H^{\frac{6+l}{6},6+l}\left([0,T]\times\sS\right)$ to \eqref{eq:linprob} and a constant $C>0$ independent of $\varphi$, $f$ and $\varphi_0$ such that 
		\begin{align}
			\Vert \varphi\Vert_{H^{\frac{6+l}{6},6+l}\left([0,T]\times\sS\right)}
			\leq C\Big( 
			\Vert f\Vert_{H^{\frac{l}{6},l}\left([0,T]\times\sS\right)}
			+ \Vert \varphi_0\Vert_{C^{6+l}(\sS)}\Big).\label{eq:estlinprob}
		\end{align}
	\end{prop}
	
	\begin{proof}
		First, assume that $a_6\equiv1$ and $a_q\equiv0$ for $q=0,1,\dots,5$. Like in \cite[Section 3]{G2019} one proves that in this case, a weak solution to \eqref{eq:linprob} exists. Localizing the problem and using \cite[Theorem VI.21]{EZ1998} one obtains Hölder regularity and \eqref{eq:estlinprob}. As in \cite[Section 3.1.4]{Ba2011}, the general linear problem \eqref{eq:linprob} is solved with the method of continuity.
	\end{proof}
	
	\subsection{The nonlinear problem}
	\label{sec:appnonlinprob}
	
	In this subsection, we proceed as in \cite[Section 2.5]{G06} and \cite[Section 3.3]{DS2017}. For a smooth immersed reference curve $\gamma_\ast$ with normal $\nu_\ast$, we define the open set
	\begin{align}
		W:=\big\lbrace  
		\varphi\in C^1([0,T]\times\sS;\RR):
		\vert\partial_x(\gamma_\ast+\varphi(t,\cdot)\nu_\ast)\vert\neq0 
		\big\rbrace
	\end{align}
	and
	$F\colon\sS\times\RR^7\times(0,\infty)\to\RR$ as
	\begin{align}
		&F\big(\cdot,\varphi,\partial_x\varphi,\dots,\partial_x^6\varphi,\vert\partial_x(\gamma_\ast+\varphi\nu_\ast)\vert^{-1}\big)
		:=
		\langle\nabla_s^2\nabla_{L^2}\sE_\lambda(\gamma_\ast+\varphi\nu_\ast),\nu_\ast^\perp\rangle\\
		&\qquad=\frac{1}{\vert\partial_x(\gamma_\ast+\varphi\,\nu_\ast)\vert^6}\,\partial_x^6\varphi+Q\big(\cdot,\varphi,\partial_x\varphi,\dots,\partial_x^5\varphi,\vert\partial_x(\gamma_\ast+\varphi\,\nu_\ast)\vert^{-1}\big),\label{eq:defP}
	\end{align}
	see \eqref{eq:defF}. Since $Q$ is a polynomial with smooth coefficients, the mapping
	\begin{align}
		\mathbf{F}\colon H^{\frac{6+l}{6},6+l}\left([0,T]\times\sS\right)\cap W&\to H^{\frac{l}{6},l}\left([0,T]\times\sS\right), \\
		\varphi&\mapsto F\big(\cdot,\varphi,\partial_x\varphi,\dots,\partial_x^6\varphi,\vert\partial_x(\gamma_\ast+\varphi\nu_\ast)\vert^{-1}\big),
	\end{align}
	$l\in\RR_+\!\setminus\ZZ$, is welldefined. With this, we write \eqref{eq:pdesystem2} as 
	\begin{align}
		\label{eq:sysF}
		\begin{cases}
			\dot\varphi=\mathbf{F}[\varphi]\quad&\text{ in } (0,T)\times\sS,\\
			\varphi(0,\cdot)=\varphi_0&\text{ on }\sS.
		\end{cases}
	\end{align}
	
	\begin{prop}\label{prop:nonlinprob}
		Take $\varphi_0\in C^{6+l}(\sS)$ such that $\gamma_\ast+\varphi_0\nu_\ast$ is an immersed curve. Then there exists $T>0$ and a unique solution $\varphi\in H^{\frac{6+l}{6},6+l}\left([0,T]\times\sS\right)$ to \eqref{eq:sysF}.
	\end{prop}
	
	\begin{proof}
		First, we assume that $0<l<1$.\smallskip
		
		\textit{Step 1: Linearized problem.} 
		Let us make the following observation. The Fréchet derivative of $\mathbf{F}$ at $\xi\in H^{\frac{6+l}{6},6+l}\left([0,T]\times\sS\right)\cap W$ can be written as
		\begin{align}
			\mathrm{D}\mathbf{F}[\xi] =\sum_{q=0}^6a_q\partial_x^q
		\end{align}
		with $a_6=\vert\partial_x(\gamma_\ast+\xi\nu_\ast)\vert^{-6}$ and $a_q(t,x)=\tilde a_q(t,x,\xi,\partial_x\xi,\dots,\partial_x^5\xi,\vert\partial_x(\gamma_\ast+\xi\nu_\ast)\vert^{-1})$, $q=0,\dots,5$, for smooth functions $\tilde a_q\colon[0,T]\times\sS\times\RR^6\times(0,\infty)\to\RR$. 
		In particular, $a_q\in H^{\frac{l}{6},l}\left([0,T]\times\sS\right)$ for $q=0,\dots,6$.
		Since $\xi\in W$ and by compactness, $a_6$ is uniformly bounded away from zero on $[0,T]\times\sS$. Hence, $\partial_t-\mathrm{D}\mathbf{F}[\xi]$ is a parabolic operator and \Cref{prop:linprob} yields a unique solution $\sigma\in H^{\frac{6+l}{6},6+l}\left([0,T]\times\sS\right)$ to 
		\begin{align}
			\begin{cases}\label{eq:sysstern}
				\dot\sigma-\mathrm{D}\mathbf{F}[\xi]\sigma=f \quad & \text{ in }[0,T]\times\sS,\\
				\sigma(0,\cdot)=\varphi_0 & \text{ on }\sS
			\end{cases}
		\end{align}
		for $f\in H^{\frac{l}{6},l}\left([0,T]\times\sS\right)$.
		\smallskip
		
		\textit{Step 2: Special solution.} Now, we consider \eqref{eq:sysstern} with $\xi\equiv0$ and $f=\mathbf{F}[\varphi_0]-\mathrm{D}\mathbf{F}[0]\varphi_0$. Here, $\varphi_0\in C^{6+l}(\sS)$ is interpreted as function $\varphi_0\in H^{\frac{6+l}{6},6+l}\left([0,T]\times\sS\right)\cap W$ constant in time. Since $f\in H^{\frac{l}{6},l}\left([0,T]\times\sS\right)$, there exists a unique solution in $ H^{\frac{6+l}{6},6+l}\left([0,T]\times\sS\right)$ to \eqref{eq:sysstern}. In the following, we will denote this solution by $\sigma$. Since $\varphi_0\in W$, we can make $T>0$ small enough such that $\sigma\in W$. Now, we define
		\begin{align}\label{eq:H}
			\tilde f:=\partial_t\sigma-\mathbf{F}[\sigma]\in H^{\frac{l}{6},l}\left([0,T]\times\sS\right).
		\end{align}
		Evaluated at $t=0$, we have
		\begin{align}
			\tilde f(0,\cdot)=\mathrm{D}\mathbf{F}[0]\sigma\vert_{t=0}+\mathbf{F}[\varphi_0]-\mathrm{D}\mathbf{F}[0]\varphi_0-\mathbf{F}[\sigma]\vert_{t=0}=0.\label{eq:ftilde(0)}
		\end{align}
		Here, we used \eqref{eq:sysstern} and the smoothness of $\mathbf{F}$.
		\smallskip
		
		\textit{Step 3: Application of the Inverse Function Theorem.} 
		Let $0<k<l$ such that $k\not\in\ZZ$ and consider the nonlinear operator 
		\begin{align}
			\Phi\colon H^{\frac{6+k}{6},6+k}\left([0,T]\times\sS\right)\cap W&\to H^{\frac{k}{6},k}\left([0,T]\times\sS\right)\times C^{6+k}(\sS),\\
			\xi&\mapsto\big(\partial_t\xi-\mathbf{F}[\xi],\,\xi(0,\cdot)\big).
		\end{align}
		By Step 1, the Fréchet derivative
		\begin{align}
			\mathrm{D}\Phi[\sigma]\colon H^{\frac{6+k}{6},6+k}\left([0,T]\times\sS\right)&\to H^{\frac{k}{6},k}\left([0,T]\times\sS\right)\times C^{6+k}(\sS),\\
			\zeta&\mapsto\big(\partial_t\zeta-\mathrm{D}\mathbf{F}[\sigma]\zeta,\,\zeta(0,\cdot)\big)
		\end{align}
		is a linear isomorphism. It follows by the Local Inverse Function Theorem that $\Phi$ is a local diffeomorphism at $\sigma$. So there is a neighborhood $U\subset H^{\frac{6+k}{6},6+k}\left([0,T]\times\sS\right)\cap W$ of $\sigma$ and a neighborhood $V\subset H^{\frac{k}{6},k}\left([0,T]\times\sS\right)\times C^{6+k}(\sS)$ of $\Phi(\sigma)=(\tilde f,\varphi_0)$ such that $\Phi\colon U\to V$ is a diffeomorphism.
		\smallskip
		
		\textit{Step 4: Cut-off in time.} Choose for $0<\varepsilon<\min\lbrace1,\frac T2\rbrace$ cut-off functions $\eta_\varepsilon\in C^\infty([0,T])$ such that $0\leq\eta_\varepsilon\leq1$, $0\leq\dot\eta_\varepsilon\leq2\varepsilon^{-1}$ and $\eta_\varepsilon\equiv0$ on $[0,\varepsilon]$ as well as $\eta_\varepsilon\equiv1$ on $[2\varepsilon,T]$. Define 
		\begin{align}
			f_\varepsilon:=\eta_\varepsilon\tilde f.\label{eq:deffvarepsilon}
		\end{align}
		Analogously as in \cite[Lemma 2.5.8]{G06}, one shows that $f_\varepsilon\in H^{\frac{l}{6},l}\left([0,T]\times\sS\right)$
		with norm bounded independently of $\varepsilon$.
		Here, we need the assumption $l<1$. The crucial point is estimating the Hölder seminorm in time. We have
		\begin{align}
			\vert f_\varepsilon(t_1)-f_\varepsilon(t_2)\vert
			\leq
			\vert\tilde f(t_1)-\tilde f(t_2)\vert+
			\vert\tilde f(t_1)\vert\vert\eta_\varepsilon(t_1)-\eta_\varepsilon(t_2)\vert
		\end{align}
		for $t_1< t_2\in[0,T]$. Assume $t_1<2\varepsilon$, otherwise $\vert\eta_\varepsilon(t_1)-\eta_\varepsilon(t_2)\vert=0$. For $t_2\leq 3\varepsilon$, we obtain with \eqref{eq:ftilde(0)} that
		\begin{align}
			\vert\tilde f(t_1)\vert\vert\eta_\varepsilon(t_1)-\eta_\varepsilon(t_2)\vert
			&=\vert\tilde f(t_1)-\tilde f(0)\vert\vert\eta_\varepsilon(t_1)-\eta_\varepsilon(t_2)\vert
			\leq \frac{C}{\varepsilon} \,t_1^\frac l6\,\vert t_1-t_2\vert\\
			&\leq \frac{C}{\varepsilon} \,t_1^\frac l6\,\vert t_1-t_2\vert^\frac l6\, t_2^{1-\frac{l}{6}}
			\leq \frac{C}{\varepsilon} \,t_2\,\vert t_1-t_2\vert^\frac l6
			\leq C \vert t_1-t_2\vert^\frac l6
		\end{align}
		for a constant $C$ not depending on $\varepsilon$.
		For $t_2>3\varepsilon$, one has
		\begin{align}
			\vert\tilde f(t_1)\vert\vert\eta_\varepsilon(t_1)-\eta_\varepsilon(t_2)\vert
			&\leq C t_1^\frac l6
			\leq C\varepsilon^\frac l6
			\leq C\vert t_1-t_2\vert^\frac l6.
		\end{align}
		Again, the constant $C$ does not depend on $\varepsilon$, which allows us to show the claim.
		Moreover, one proves with standard arguments (see for example \cite[Section 6.8]{GT2001}) and the Arzelà--Ascoli theorem, that for $0<k<l<1$, $H^{\frac{l}{6},l}\left([0,T]\times\sS\right)$ is compactly embedded in $H^{\frac{k}{6},k}\left([0,T]\times\sS\right)$.
		It follows that there exists a sequence $\varepsilon_n\to0$, $n\to\infty$, and $\hat f\in H^{\frac{l}{6},l}\left([0,T]\times\sS\right)$ such that 
		\begin{align}\label{eq:limfepsilonn}
			f_{\varepsilon_n}\to\hat f \;\text{ in }H^{\frac{k}{6},k}\left([0,T]\times\sS\right)\;\text{ for } n\to\infty.
		\end{align}
		Due to \eqref{eq:ftilde(0)} and the continuity of $\tilde f$, we find that $f_\varepsilon\to\tilde f$, $\varepsilon\to0$ uniformly in the sup norm. So we conclude that $\hat f=\tilde f$. 
		\smallskip
		
		\textit{Step 5: Existence of solution. }
		By \eqref{eq:limfepsilonn}, there exists some $\varepsilon_\ast>0$ small enough such that $(f_{\varepsilon_\ast},\varphi_0)\in V$ with $V$ given as in Step 3. This implies the existence of a unique $\varphi\in U$ such that 
		\begin{align}
			\partial_t\varphi=\mathbf{F}[\varphi]+f_{\varepsilon_\ast}\quad\text{and}\quad\varphi(0,\cdot)=\varphi_0.
		\end{align}
		With \eqref{eq:deffvarepsilon} and the definition of $\eta_\varepsilon$, we conclude that $\varphi$ solves \eqref{eq:sysF} for $t\in[0,\varepsilon_\ast]$.
		\smallskip
		
		\textit{Step 6: Regularity. } So far, we have $\varphi\in H^{\frac{6+k}{6},6+k}\left([0,\varepsilon_\ast]\times\sS\right)\cap W$ for $0<k<l<1$. Defining 
		\begin{align}
			a_6(t,x):=\vert\partial_x(\gamma_\ast+\varphi\nu_\ast)\vert^{-6}\in H^{\frac{5+k}{6},5+k}\left([0,\varepsilon_\ast]\times\sS\right)\subset H^{\frac{l}{6},l}\left([0,\varepsilon_\ast]\times\sS\right)
		\end{align}
		and
		\begin{align}
			p(t,x):=P(\cdot,\varphi,\dots,\partial_x^5\varphi,\vert\partial_x(\gamma_\ast+\varphi\nu_\ast)\vert^{-1})\in H^{\frac{1+k}{6},1+k}\left([0,\varepsilon_\ast]\times\sS\right)\subset H^{\frac{l}{6},l}\left([0,\varepsilon_\ast]\times\sS\right)
		\end{align}
		with $P$ as in \eqref{eq:defP}, \Cref{prop:linprob} yields a unique solution $\xi\in H^{\frac{6+l}{6},6+l}\left([0,\varepsilon_\ast]\times\sS\right)$ to 
		\begin{align}\label{eq:probregularity}
			\begin{cases}
				\partial_t\xi-a_6\partial_x^6\xi=p \quad&\text{ in }[0,\varepsilon_\ast]\times\sS,\\
				\xi(0,\cdot)=\varphi_0 &\text{ on }\sS.
			\end{cases}
		\end{align}
		Since $\varphi$ solves \eqref{eq:probregularity} on $[0,\varepsilon_\ast]\times\sS$, we conclude that $\xi=\varphi$ for $t\leq\varepsilon_\ast$. Hence, we obtain $\varphi\in H^{\frac{6+l}{6},6+l}\left([0,\varepsilon_\ast]\times\sS\right)$.
		\smallskip
		
		\textit{Step 7: Uniqueness.}
		Assume we have two solutions $\varphi_1,\varphi_2\in H^{\frac{6+l}{6},6+l}\left([0,\varepsilon_i]\times\sS\right)$, $i=1,2$ and $\varepsilon_1\leq\varepsilon_2\leq\varepsilon_\ast$. Define
		\begin{align}
			T^\ast=\sup\big\lbrace t\in[0,\varepsilon_1)\colon\varphi_1=\varphi_2\text{ in }C^{6+l}(\sS)\text{ on }[0,t]\big\rbrace.
		\end{align}
		We consider $\Phi$ of Step 3 for variable $T$ and show that for $T>0$ small enough, 
		\begin{align}
			\label{eq:inUinV}
			\varphi_1\vert_{[0,T]},\varphi_2\vert_{[0,T]}\in U \quad\text{ as well as }\quad (0,\varphi_0)\in V.
		\end{align}
		From the proof of the Inverse Function Theorem, 
		it follows that there exists $r>0$ such that 
		\begin{align}\label{eq:uniformnghb}
			B_r(\sigma)\subset U(T)\quad\text{ and }\quad B_r(\tilde f,\varphi_0)\subset V(T)\quad\text{ for all }0<T\leq\varepsilon_1.
		\end{align}
		Here, $\sigma$ and $\tilde f$ are given as in Step 2. Now, we first note that with \eqref{eq:H} and \eqref{eq:ftilde(0)}, for $q=0,\dots,6$,
		\begin{align}
			\partial_x^q(\varphi_i-\sigma)\vert_{t=0}=0\quad\text{ and }\quad\partial_t(\varphi_i-\sigma)\vert_{t=0}=\big(\mathbf{F}[\varphi_i]-\mathbf{F}[\sigma]-\tilde f\big)\vert_{t=0}=0.
		\end{align}
		With \Cref{lem:normto0} it follows that 
		\begin{align}
			\Vert\varphi_i-\sigma\Vert_{H^{\frac{6+k}{6},6+k}\left([0,T]\times\sS\right)}\to0\quad \text{ as } T\to0.
		\end{align}
		Similarly, \eqref{eq:ftilde(0)} implies
		\begin{align}
			\Vert\tilde f\Vert_{H^{\frac{k}{6},k}\left([0,T]\times\sS\right)}\to0\quad \text{ as } T\to0.
		\end{align}
		Due to \eqref{eq:uniformnghb} we conclude that for $T>0$ small enough, \eqref{eq:inUinV} holds and thus
		\begin{align}
			\varphi_1=\varphi_2 \text{ in }{H^{\frac{6+k}{6},6+k}\left([0,T]\times\sS\right)}.
		\end{align}
		The regularity of $\varphi_1$ and $\varphi_2$ even gives coincidence in ${H^{\frac{6+l}{6},6+l}\left([0,T]\times\sS\right)}$. It follows that $T^\ast>0$. 
		Now, assume $T^\ast<\varepsilon_1$. Then $\varphi_1(T^\ast)=\varphi_2(T^\ast)=:\tilde\varphi_0\in C^{6+l}(\sS)$. Taking $\tilde\varphi_0$ as new initial datum and repeating the arguments yields a contradiction to the definition of $T^\ast$. Hence, we get $T^\ast=\varepsilon_1$ and uniqueness is proven.
		\smallskip
		
		Bootstrapping and using the same arguments as in Step 6 shows the claim for $l>1$.
	\end{proof}
	
	In the last step of the proof of \Cref{prop:nonlinprob}, we use the following observation. 
	
	\begin{lem}\label{lem:normto0}
		Let $0<k<l<1$ and $f\in H^{\frac{6+l}{6},6+l}\left([0,T]\times\sS\right)$ such that 
		\begin{align}
			f(0,\cdot)=\partial_tf(0,\cdot)=\partial_xf(0,\cdot)=\dots=\partial_x^6f(0,\cdot)=0.
		\end{align}
		Then 
		$
		\Vert f\Vert_{H^{\frac{6+k}{6},6+k}\left([0,T]\times\sS\right)}\to0$ as $T\to0.
		$
	\end{lem}
	\begin{proof}
		Consider the periodic counterpart $f\colon[0,T]\times\RR\to\RR$ and let the interval $I\subset\RR$ contain at least two periods of $f$. Then 
		\begin{align}
			&\Vert f\Vert_{H^{\frac{6+k}{6},6+k}\left([0,T]\times\sS\right)}
			=\sup_{[0,T]\times I}\vert f\vert+\sup_{[0,T]\times I}\vert \partial_tf\vert+\sup_{[0,T]\times I}\vert \partial_xf\vert+\dots+\sup_{[0,T]\times I}\vert \partial_x^6f\vert\\
			&\qquad +\sup_{\substack{t\in[0,T],x_1,x_2\in I\\x_1\neq x_2}}\frac{\vert\partial_tf(t,x_1)-\partial_tf(t,x_2)\vert}{\vert x_1-x_2\vert^k}
			+\sup_{\substack{t\in[0,T],x_1,x_2\in I\\x_1\neq x_2}}\frac{\vert\partial_x^6f(t,x_1)-\partial_x^6f(t,x_2)\vert}{\vert x_1-x_2\vert^k}\\
			&\qquad +\sup_{\substack{t_1,t_2\in[0,T],x\in I\\t_1\neq t_2}}\frac{\vert\partial_tf(t_1,x)-\partial_tf(t_2,x)\vert}{\vert t_1-t_2\vert^\frac k6}
			+\sup_{\substack{t_1,t_2\in[0,T],x\in I\\t_1\neq t_2}}\frac{\vert\partial_xf(t_1,x)-\partial_xf(t_2,x)\vert}{\vert t_1-t_2\vert^\frac{5+k}{6}}\\
			&\qquad+\dots+\sup_{\substack{t_1,t_2\in[0,T],x\in I\\t_1\neq t_2}}\frac{\vert\partial_x^6f(t_1,x)-\partial_x^6f(t_2,x)\vert}{\vert t_1-t_2\vert^\frac k6}.
		\end{align}
		For the terms coming from the Hölder seminorm in space, interpolation as in \cite[Prop. 1.1.3]{L1995} yields
		\begin{align}
			\sup_{\substack{t\in[0,T],x_1,x_2\in I\\x_1\neq x_2}}\frac{\vert\partial_tf(t,x_1)-\partial_tf(t,x_2)\vert}{\vert x_1-x_2\vert^k}
			&\leq \sup_{t\in[0,T]}\Vert\partial_tf\Vert_{C^k(\sS)}\\
			&\leq \sup_{t\in[0,T]}\Vert\partial_tf\Vert^{1-\frac kl}_{C^l(\sS)}\Vert\partial_tf\Vert^\frac kl_{C(\sS)}\to 0\quad\text{ as }T\to0.
		\end{align}
		For the terms coming from the Hölder seminorm in time, we have
		\begin{align}
			\sup_{\substack{t_1,t_2\in[0,T],x\in I\\t_1\neq t_2}}\frac{\vert\partial_tf(t,x_1)-\partial_tf(t,x_2)\vert}{\vert t_1-t_2\vert^\frac k6}
			&=\sup_{\substack{t_1,t_2\in[0,T],x\in I\\t_1\neq t_2}}\frac{\vert\partial_tf(t,x_1)-\partial_tf(t,x_2)\vert}{\vert t_1-t_2\vert^\frac l6}\vert t_1-t_2\vert^\frac{l-k}{6}\\
			&\leq C\sup_{\substack{t_1,t_2\in[0,T]\\t_1\neq t_2}}\vert t_1-t_2\vert^\frac{l-k}{6}\to 0\quad\text{ as }T\to0.
		\end{align}
		Treating the other terms with the same arguments yields the claim.
	\end{proof}

	
	\section*{Acknowledgements}
	The author would like to thank Anna Dall'Acqua for constructive discussions and valuable comments, as well as Fabian Rupp for his helpful feedback.
	
	\bibliographystyle{abbrv}
	\bibliography{biblio}

\begin{thebibliography}{10}

\bibitem{ADG17}
J.-J. Alibert, A.~Della~Corte, I.~Giorgio, and A.~Battista.
\newblock Extensional {\it {e}lastica} in large deformation as {$\Gamma$}-limit
  of a discrete 1{D} mechanical system.
\newblock {\em Z. Angew. Math. Phys.}, 68(2):Paper No. 42, 19, 2017.

\bibitem{AMCWW2020}
B.~Andrews, J.~McCoy, G.~Wheeler, and V.-M. Wheeler.
\newblock Closed ideal planar curves.
\newblock {\em Geom. Topol.}, 24(2):1019--1049, 2020.

\bibitem{Ba2011}
C.~Baker.
\newblock The mean curvature flow of submanifolds of high codimension.
\newblock {\em arXiv}, 1104.4409, 2011.

\bibitem{BL2012}
J.~Bergh and J.~L{\"o}fstr{\"o}m.
\newblock {\em Interpolation Spaces: An Introduction}.
\newblock Grundlehren der mathematischen Wissenschaften. Springer Berlin
  Heidelberg, 2012.

\bibitem{Blatt}
S.~Blatt.
\newblock Loss of convexity and embeddedness for geometric evolution equations
  of higher order.
\newblock {\em J. Evol. Equ.}, 10(1):21--27, 2010.

\bibitem{BuHe17}
D.~Bucur and A.~Henrot.
\newblock A new isoperimetric inequality for elasticae.
\newblock {\em J. Eur. Math. Soc. (JEMS)}, 19(11):3355--3376, 2017.

\bibitem{BGH1998}
G.~Buttazzo, M.~Giaquinta, and S.~Hildebrandt.
\newblock {\em One-dimensional Variational Problems: An Introduction}.
\newblock Oxford lecture series in mathematics and its applications. Clarendon
  Press, 1998.

\bibitem{CFS09}
R.~Chill, E.~Fa\v{s}angov\'{a}, and R.~Sch\"{a}tzle.
\newblock Willmore blowups are never compact.
\newblock {\em Duke Math. J.}, 147(2):345--376, 2009.

\bibitem{DALP2014}
A.~Dall'Acqua, C.-C. Lin, and P.~Pozzi.
\newblock Evolution of open elastic curves in {$\Bbb{R}^n$} subject to fixed
  length and natural boundary conditions.
\newblock {\em Analysis (Berlin)}, 34(2):209--222, 2014.

\bibitem{DALP2019}
A.~Dall'Acqua, C.-C. Lin, and P.~Pozzi.
\newblock Elastic flow of networks: long-time existence result.
\newblock {\em Geom. Flows}, 4(1):83--136, 2019.

\bibitem{DP2017}
A.~Dall'Acqua and A.~Pluda.
\newblock Some minimization problems for planar networks of elastic curves.
\newblock {\em Geom. Flows}, 2(1):105--124, 2017.

\bibitem{DP14}
A.~Dall'Acqua and P.~Pozzi.
\newblock A {W}illmore-{H}elfrich {$L^2$}-flow of curves with natural boundary
  conditions.
\newblock {\em Comm. Anal. Geom.}, 22(4):617--669, 2014.

\bibitem{DPS16}
A.~Dall'Acqua, P.~Pozzi, and A.~Spener.
\newblock The {{\L}}ojasiewicz-{S}imon gradient inequality for open elastic
  curves.
\newblock {\em J. Differential Equations}, 261(3):2168--2209, 2016.

\bibitem{DS2017}
A.~Dall'Acqua and A.~Spener.
\newblock The elastic flow of curves in the hyperbolic plane.
\newblock {\em arXiv}, 1710.09600, 2017.

\bibitem{DHMV2008}
P.~A. Djondjorov, M.~T. Hadzhilazova, I.~M. Mladenov, and V.~M. Vassilev.
\newblock Explicit parameterization of {Euler’s} elastica.
\newblock {\em Geometry, integrability and quantization}, pages 175 -- 186,
  2008.

\bibitem{DKS2002}
G.~Dziuk, E.~Kuwert, and R.~Sch\"{a}tzle.
\newblock Evolution of elastic curves in {$\Bbb R^n$}: existence and
  computation.
\newblock {\em SIAM J. Math. Anal.}, 33(5):1228--1245, 2002.

\bibitem{EZ1998}
S.~Eidelman and N.~Zhitarashu.
\newblock {\em Parabolic Boundary Value Problems}.
\newblock Operator Theory: Advances and Applications. Birkh{\"a}user Basel,
  1998.

\bibitem{EGW18}
M.~I. Espa\~nol, D.~Golovaty, and J.~P. Wilber.
\newblock Euler elastica as a {$\Gamma$}-limit of discrete bending energies of
  one-dimensional chains of atoms.
\newblock {\em Math. Mech. Solids}, 23(7):1104--1116, 2018.

\bibitem{F2000}
P.~C. Fife.
\newblock Models for phase separation and their mathematics.
\newblock {\em Electron. J. Differential Equations}, pages No. 48, 26, 2000.

\bibitem{G86}
M.~Gage.
\newblock On an area-preserving evolution equation for plane curves.
\newblock In {\em Nonlinear problems in geometry ({M}obile, {A}la., 1985)},
  volume~51 of {\em Contemp. Math.}, pages 51--62. Amer. Math. Soc.,
  Providence, RI, 1986.

\bibitem{GG21}
H.~Garcke and M.~Gößwein.
\newblock Non-linear stability of double bubbles under surface diffusion.
\newblock {\em Journal of Differential Equations}, 302:617--661, 2021.

\bibitem{GMP2020}
H.~Garcke, J.~Menzel, and A.~Pluda.
\newblock Long time existence of solutions to an elastic flow of networks.
\newblock {\em Comm. Partial Differential Equations}, 45(10):1253--1305, 2020.

\bibitem{G06}
C.~Gerhardt.
\newblock {\em Curvature Problems}.
\newblock Series in geometry and topology. International Press, 2006.

\bibitem{GT2001}
D.~Gilbarg and N.~Trudinger.
\newblock {\em Elliptic Partial Differential Equations of Second Order}.
\newblock Classics in Mathematics. Springer, Berlin Heidelberg, 2001.

\bibitem{G2019}
M.~G{\"o}{\ss}wein.
\newblock Surface diffusion flow of triple junction clusters in higher space
  dimensions.
\newblock {\em PhD thesis, University of Regensburg}, 2019.

\bibitem{K24}
E.~Kim and D.~Kwon.
\newblock Area-preserving anisotropic mean curvature flow in two dimensions.
\newblock {\em Calc. Var. Partial Differential Equations}, 64(1):Paper No. 27,
  39, 2025.

\bibitem{K1996}
N.~Koiso.
\newblock On the motion of a curve towards elastica.
\newblock In {\em Actes de la {T}able {R}onde de {G}\'{e}om\'{e}trie
  {D}iff\'{e}rentielle ({L}uminy, 1992)}, volume~1 of {\em S\'{e}min. Congr.},
  pages 403--436. Soc. Math. France, Paris, 1996.

\bibitem{LSU1988}
O.~A. Lady\v{z}enskaja, V.~A. Solonnikov, and N.~N. Ural'ceva.
\newblock {\em Linear and Quasi-linear Equations of Parabolic Type}.
\newblock American Mathematical Society, United States of America, reprinted
  edition edition, 1988.

\bibitem{LS1984}
J.~Langer and D.~A. Singer.
\newblock {The total squared curvature of closed curves}.
\newblock {\em Journal of Differential Geometry}, 20(1):1 -- 22, 1984.

\bibitem{L08}
R.~Levien.
\newblock The elastica: a mathematical history.
\newblock {\em Tech. Report UCB/EECS-2008-103, EECS Department, University of
  California, Berkeley}, 2008.

\bibitem{Lin}
C.-C. Lin.
\newblock {$L^2$}-flow of elastic curves with clamped boundary conditions.
\newblock {\em J. Differential Equations}, 252(12):6414--6428, 2012.

\bibitem{LL2018}
C.-C. Lin and Y.-K. Lue.
\newblock Evolving inextensible and elastic curves with clamped ends under the
  second-order evolution equation in {$\Bbb R^2$}.
\newblock {\em Geom. Flows}, 3(1):14--18, 2018.

\bibitem{LL2015}
C.-C. Lin, Y.-K. Lue, and H.~R. Schwetlick.
\newblock The second-order {$L^2$}-flow of inextensible elastic curves with
  hinged ends in the plane.
\newblock {\em J. Elasticity}, 119(1-2):263--291, 2015.

\bibitem{L1995}
A.~Lunardi.
\newblock {\em Analytic Semigroups and Optimal Regularity in Parabolic
  Problems}.
\newblock Birkhäuser Verlag, Basel, 1995.

\bibitem{MR4277362}
C.~Mantegazza, A.~Pluda, and M.~Pozzetta.
\newblock A survey of the elastic flow of curves and networks.
\newblock {\em Milan J. Math.}, 89(1):59--121, 2021.

\bibitem{MantegazzaPozzetta21}
C.~Mantegazza and M.~Pozzetta.
\newblock The {{\L}}ojasiewicz-{S}imon inequality for the elastic flow.
\newblock {\em Calc. Var. Partial Differential Equations}, 60(1):Paper No. 56,
  17, 2021.

\bibitem{MW22}
J.~A. McCoy, G.~E. Wheeler, and Y.~Wu.
\newblock A length-constrained ideal curve flow.
\newblock {\em Q. J. Math.}, 73(2):685--699, 2022.

\bibitem{MMR2021}
T.~Miura, M.~M\"uller, and F.~Rupp.
\newblock Optimal thresholds for preserving embeddedness of elastic flows.
\newblock {\em Amer. J. Math.}, 147(1):33--80, 2025.

\bibitem{MW24}
T.~Miura and G.~Wheeler.
\newblock The free elastic flow for closed planar curves.
\newblock {\em arXiv}, 2404.12619, 2024.

\bibitem{MR2021}
M.~M\"{u}ller and F.~Rupp.
\newblock A {L}i-{Y}au inequality for the 1-dimensional {W}illmore energy.
\newblock {\em Adv. Calc. Var.}, 16(2):337--362, 2023.

\bibitem{NO2014}
M.~Novaga and S.~Okabe.
\newblock Curve shortening-straightening flow for non-closed planar curves with
  infinite length.
\newblock {\em J. Differential Equations}, 256(3):1093--1132, 2014.

\bibitem{O2007}
S.~Okabe.
\newblock The motion of elastic planar closed curves under the area-preserving
  condition.
\newblock {\em Indiana Univ. Math. J.}, 56(4):1871--1912, 2007.

\bibitem{PR09}
M.~A. Peletier and M.~R\"oger.
\newblock Partial localization, lipid bilayers, and the elastica functional.
\newblock {\em Arch. Ration. Mech. Anal.}, 193(3):475--537, 2009.

\bibitem{P1996}
A.~Polden.
\newblock Curves and surfaces of least total curvature and fourth-order flows.
\newblock {\em PhD thesis, Universität Tübingen}, 1996.

\bibitem{ConstrLoja}
F.~Rupp.
\newblock On the {{\L}}ojasiewicz-{S}imon gradient inequality on submanifolds.
\newblock {\em J. Funct. Anal.}, 279(8):108708, 33, 2020.

\bibitem{RS2020}
F.~Rupp and A.~Spener.
\newblock Existence and convergence of the length-preserving elastic flow of
  clamped curves.
\newblock {\em J. Evol. Equ.}, 24(3):Paper No. 59, 2024.

\bibitem{Z97}
Z.~Shao-guang.
\newblock A complete classification of closed shapes for cylindrical vesicles.
\newblock {\em Acta Phys. Sin. (Overseas Edn)}, 6(9):641, sep 1997.

\bibitem{VDM08}
V.~M. Vassilev, P.~A. Djondjorov, and I.~M. Mladenov.
\newblock Cylindrical equilibrium shapes of fluid membranes.
\newblock {\em Journal of Physics A: Mathematical and Theoretical},
  41(43):435201, sep 2008.

\bibitem{W23}
X.-L. Wang.
\newblock The evolution of area-preserving and length-preserving inverse
  curvature flows for immersed locally convex closed plane curves.
\newblock {\em J. Funct. Anal.}, 284(1):Paper No. 109744, 25, 2023.

\bibitem{W2000}
K.~Watanabe.
\newblock Plane domains which are spectrally determined.
\newblock {\em Ann. Global Anal. Geom.}, 18(5):447--475, 2000.

\bibitem{WT2008}
K.~Watanabe and I.~Takagi.
\newblock Representation formula for the critical points of the
  {T}adjbakhsh-{O}deh functional and its application.
\newblock {\em Japan J. Indust. Appl. Math.}, 25(3):331--372, 2008.

\bibitem{W1993}
Y.~Wen.
\newblock {$L^2$ flow of curve straightening in the plane}.
\newblock {\em Duke Mathematical Journal}, 70(3):683 -- 698, 1993.

\bibitem{W1965}
E.~F. Whittlesey.
\newblock Analytic {F}unctions in {B}anach {S}paces.
\newblock {\em Proceedings of the American Mathematical Society},
  16(5):1077--1083, 1965.

\bibitem{W21}
Y.~Wu.
\newblock Gradient flow of the dirichlet energy for the curvature of plane
  curves.
\newblock {\em Doctor of Philosophy thesis, School of Mathematics and Applied
  Statistics, University of Wollongong}, 2021.

\end{thebibliography}
	
\end{document}